\documentclass[a4paper,10pt,reqno]{amsart}

\usepackage{amssymb,amsmath,amsfonts} 
\usepackage{a4wide}
\usepackage{graphicx}
\usepackage{hyperref}
\usepackage{color}
\usepackage{amsthm}
\usepackage{enumerate}
\numberwithin{equation}{section}   
\theoremstyle{plain}
\newtheorem{Theorem}{Theorem}[section]
\newtheorem{Proposition}[Theorem]{Proposition}
\newtheorem{Corollary}[Theorem]{Corollary}
\newtheorem{Lemma}[Theorem]{Lemma}
\theoremstyle{definition}

\newtheorem{Definition}[Theorem]{Definition}
\newtheorem{Claim}[Theorem]{Claim}

\theoremstyle{remark}
\newtheorem{Remark}[Theorem]{Remark}

\newcommand{\one}{\mathbf{1}} 

\newcommand{\FF}{\mathbb{F}}
\newcommand{\RR}{\mathbb{R}}

\newcommand{\NN}{\mathbb{N}}

\newcommand{\ZZ}{\mathbb{Z}}

\newcommand{\cE}{\mathcal{E}}
\newcommand{\cF}{\mathcal{F}}
\newcommand{\cP}{\mathcal{P}}
\newcommand{\cD}{\mathcal{D}}
\newcommand{\cB}{\mathcal{B}}

\newcommand{\cS}{\mathcal{S}}

\newcommand{\cX}{\mathcal{X}}

\newcommand{\cO}{\mathcal{O}}

\newcommand{\cQ}{\mathcal{Q}}
\newcommand{\cR}{\mathcal{R}}

\newcommand{\cI}{\mathcal{I}}
\newcommand{\cT}{\mathcal{T}}

\newcommand\abs[1]{\left|#1\right|}

\newcommand\set[1]{\left\{{#1}\right\}}

\newcommand{\hm}[1]{\textbf{*}\leavevmode{\marginpar{\tiny%
$\hbox to 0mm{\hspace*{-0.5mm}$\leftarrow$\hss}%
\vcenter{\vrule depth 0.1mm height 0.1mm width \the\marginparwidth}%
\hbox to 0mm{\hss$\rightarrow$\hspace*{-0.5mm}}$\\\relax\raggedright #1}}}
\title[Shannon-McMillan-Breiman theorem along geodesics]{Shannon-McMillan-Breiman theorem \\ along almost geodesics \\ in negatively curved groups}
\author{Amos Nevo}
\address{Amos Nevo : Dept. of Mathematics, Technion, and Dept. of Mathematics, University of Chicago.}
\email{anevo@technion.ac.il, nevo@uchicago.edu}

\author{Felix Pogorzelski}
\address{Felix Pogorzelski:  Dept. of Mathematics and Computer Science, University of Leipzig.}
\email{felix.pogorzelski@math.uni-leipzig.de}
\thanks{Both authors gratefully acknowledge partial support through grant no.\@ I-1485-304.6/2019 by the German Israeli Foundation for Scientific Research and Development (GIF)}
\begin{document}

\begin{abstract}
Consider a non-elementary Gromov-hyperbolic group $\Gamma$ with a suitable invariant hyperbolic metric, and an ergodic probability measure preserving (p.m.p.) action on $(X,\mu)$. We construct special increasing sequences of finite subsets $F_n(y)\subset \Gamma$, with $(Y,\nu)$ a suitable probability space, with the following properties.
\begin{itemize}
\item Given any countable partition $\cP$ of $X$ of finite Shannon entropy, the refined partitions $\bigvee_{\gamma\in F_n(y)}\gamma \cP$ have normalized information functions which converge to a constant limit, for $\mu$-almost every $x\in X$ and $\nu$-almost every $y\in Y$. 
\item The sets $\cF_n(y)$ constitute almost-geodesic segments, and $\bigcup_{n\in \NN} F_n(y)$ is a one-sided almost geodesic with limit point $F^+(y)\in \partial \Gamma$, starting at a fixed bounded distance from the identity, 
for almost every $y\in Y$. 
\item The distribution of the limit point $F^+(y)$ belongs to the Patterson-Sullivan measure class on $\partial \Gamma$ associated with the invariant hyperbolic metric.
\end{itemize}
The main result of the present paper amounts therefore to a Shannon-McMillan-Breiman theorem along almost geodesic segments in any p.m.p. action of $\Gamma$ as above. 
For several important classes of examples we analyze, the construction of $F_n(y)$ is  purely geometric and explicit. 

Furthermore, consider the infimum of the limits of the normalized information functions,  taken over all $\Gamma$-generating partitions of $X$. Using an important inequality due to B. Seward \cite{Se16}, we deduce that it is equal to the Rokhlin entropy $\frak{h}^{\text{Rok}}$ of the $\Gamma$-action on $(X,\mu)$ defined in \cite{Se15a}, provided that the action is free.
Remarkably, this property holds for every choice of invariant hyperbolic metric, every choice of suitable auxiliary space $(Y,\nu)$, and every choice of special family $F_n(y)$ as above. In particular, for every $\epsilon > 0$, there is a generating partition $\cP_\epsilon$,  such that for almost every $y\in Y$, the partition refined using the sets $F_n(y)$ has most of its atoms of roughly constant measure,  comparable to $\exp (-n\frak{h}^{\text{Rok}}\pm \epsilon)$. This describes an approximation to the Rokhlin entropy in geometric and dynamical terms, for actions of word-hyperbolic groups. 
\end{abstract}

\maketitle

{\bf Keywords:} Rokhlin entropy, orbital entropy, skew transformations, entropy equipartition theorems, almost geodesics, p.m.p.\@ action of countable groups, hyperbolic groups.\\ 


\section{Introduction}

\subsection{Shannon-McMillan-Breiman theorem for measure-preserving transformations}

Let $S$ be an invertible probability-measure-preserving (p.m.p.) transformation on a standard probability space $(X,\cB, \mu)$. Given a countable partition of $\cP=\set{A_i,i\in \NN}$ of $X$ into measurable subsets of positive measure, the translated partition is defined by $S^{-1}\cP=\set{S^{-1}A_i, i\in \NN}$. Given another partition $\cQ=\set{B_j,j\in \NN}$, consider the mutual refinement of $\cP$ and $\cQ$ given by $\cP\vee \cQ=\set{A_i\cap B_j\neq \emptyset, i,j\in \NN}$. The {\it Shannon entropy} of the partition $\cP$ is defined by $H(\cP)=-\sum_{i\in \NN} \mu(A_i) \log \mu(A_i)\ge 0$. 

To study the dynamics of a partition under the transformation $S$, fix a countable partition $\cP$ of finite Shannon entropy and consider the  sequence of refined partitions :
$$\cP_n=\cP\vee S^{-1}\cP\vee S^{-2}\cP\vee \cdots \vee S^{-n}\cP\,,$$
and the associated sequence $H(\cP_n)$ of (finite) Shannon entropies. The sequence $H(\cP_n)$ being subadditive, it follows that the sequence of normalized Shannon entropies converges to a limit:
$$\lim_{n\to \infty}\frac{1}{n+1} H(\cP_n)=\lim_{n\to \infty}\frac{1}{n+1}  H\left(\bigvee_{k=0}^n S^{-k}\cP\right):=h_\cP(S,X),$$
  called the {\it entropy of the partition $\cP$} under the transformation $S$. Note that we can also apply this construction to $S^{-1}$, and the limit will be the same. 
  
 A more refined gauge of the dynamics of entropy is the {\it information function $\cI_\cP$}  associated with the partition $\cP$, given by 
$$\cI_{\cP}(x)=-\log \mu(A_{i(x)})=-\log \mu(\cP(x)),$$
where $A_{i(x)}=\cP(x)$ is the unique element of the partition containing $x$.
Clearly $\int_X \cI_\cQ(x)d\mu(x)=H(\cQ)$ for every countable partition $\cQ$. 
The Shannon-McMillan-Breiman theorem states that if the action of $S$ on $(X,\mu)$ is ergodic, then 
for $\mu$-almost every $x\in X$ 
$$\lim_{n\to \infty} \frac{1}{n+1}\cI_{\cP_n}(x)=h_\cP(S,X).$$
Thus the Shannon-McMillan-Breiman theorem is a pointwise almost sure convergence result for the sequence of normalized information functions $\cI_{\cP_n}(x)/(n+1)$. Since $\int_X \cI_{\cP_n}(x)d\mu(x)=H(\cP_n)$, it implies (under suitable conditions) the mean convergence result for the sequence of normalized Shannon entropies of the refined partitions stated, namely the Shannon-McMillan theorem.

 \subsection{Shannon-McMillan-Breiman theorem for amenable groups} 
  The preceding discussion motivated the goal of establishing convergence theorems for the Shannon entropy of refined partitions for actions of more general groups.  
 Let $(X,\cB,\mu)$ be a standard Borel probability space,  with $\mu$ a Borel measure on $(X,\cB)$ and let $\Gamma$ be any countable group of Borel bijections $\gamma : X \to X$ which preserve $\mu$. For a partition $\cP=\set{A_i,i\in \NN}$ of $X$ and a p.m.p. invertible map $w: X\to X$ let $w \cP=\set{wA_i, i\in \NN}$. Given a family of finite sets $F_t\subset \Gamma$, where $t\in \NN$ or $t\in \RR_+$, and a Borel measurable countable partition $\cP$ of $X$, consider the following definitions. 
 
 \medskip
 
 {\bf Entropy information functions of refined partitions. } The partition $\cP$ refined by $F_t$ is defined by 
\begin{equation}\label{def-info-func}
\cP^{F_t}=\bigvee_{\gamma\in F_t} \gamma \cP \,,\ \text{  with information function } \cI_{\cP^{F_t}}(x)=-\log \mu(\cP^{F_t}(x))\end{equation}
where $\cP^{F_t}(x)$ is the unique atom of $\cP^{F_t}$ containing $x$. 
Convergence of entropy for the family $\set{F_t}$ should assert that for any p.m.p. action of $\Gamma$ and a partition $\cP$ with finite Shannon entropy, the following limit exists:
\begin{equation}\label{def-shannon-ent}
\lim_{t\to \infty} \frac{1}{\abs{F_t}} H(\cP^{F_t})\,.
\end{equation}
The Shannon-McMillan-Breiman theorem should assert that when the action is ergodic then pointwise almost sure convergence holds for the information functions associated with the refined partitions, namely 
for $\mu$-almost every $x\in X$ :
\begin{equation}\label{def-SMB}
\lim_{t\to \infty} \frac{1}{\abs{F_t}}\cI_{\cP^{F_t}}(x)=\lim_{t\to \infty} \frac{1}{\abs{F_t}} H(\cP^{F_t})\,.
\end{equation}
Finally, the Shannon-McMillan theorem should assert that the sequence of information functions converge in the mean (namely in $L^1$ and $L^2$ norm) to the limit. 

When $\Gamma$ is amenable, a natural course of action is to choose $F_t$ as a left asymptotically invariant sequence.  A crucial point is that (under suitable conditions) the limit in question {\it does not depend} on the asymptotically invariant family $\set{F_t}$ chosen to effect the refinement process $\cP^{F_t}$. A succession of conditions on a sequence of sets  $\set{F_n}$ were used to prove convergence theorems for refinements along asymptotically invariant sequences. These results include: 
\begin{itemize}
\item  Convergence of entropy 
 for  
 general sequences (Kieffer  \cite{Ki75} and Ollagnier \cite{Ol85}),   
 \item Pointwise convergence 
 for regular sequences (Ornstein-Weiss \cite{OW83} and \cite{OW87}), 
 \item  
 Pointwise convergence for tempered sequences (Lindenstrauss \cite{Li01} and Weiss \cite{We03}).  
\end{itemize}

The successful theory of entropy convergence theorems for p.m.p. actions of  amenable groups naturally raises the problem of extending such result to non-amenable groups. 
Let us note however that even the entropy convergence result stated in (\ref{def-shannon-ent}) and certainly the mean and the pointwise almost sure convergence results stated in (\ref{def-SMB}) have not yet been established for {\it any fixed sequence} of sets $F_t$ on any non-amenable group $\Gamma$ (unless the sets are supported on an amenable subgroup). In particular no such result has ever been established on any sequence of metric balls for any left-invariant metric on any non-amenable group.

\subsection{Convergence along variable finite sets}

Given the state of affairs just noted, the following approach to proving some useful version of the Shannon-McMillan-Breiman theorem seems to be very natural. 

Suppose that $(Y,\nu)$ is an (auxiliary) standard Borel space with a probability measure $\nu$, and that there exists a family of measurable maps $F_t(\cdot) :Y \to \text{Fin}(\Gamma)$ (for $t\in \NN$ or $t\in \RR_+$), where $\text{Fin}(\Gamma)$ is the space of finite non-empty subsets of $\Gamma$. 
We can then consider the refined partitions :
\begin{equation}\label{refine}
\cP^{F_t(y)}=\bigvee_{\gamma\in F_t(y)}\gamma \cP
\end{equation} 
and the information functions they define, and study their convergence properties, to which we now turn.

\section{Main results} 
In the present paper we will concentrate on the class of negatively curved groups, and more generally, Gromov-hyperbolic groups, with a given invariant hyperbolic metric. Our  goal is to construct an auxiliary space $(Y,\nu)$ and maps $F_t$ such that the family  $\set{F_t(y)}_{y\in Y}$ satisfies the Shannon-McMillan-Breiman pointwise convergence theorem stated above,  for any given action of $\Gamma$, for almost every $y\in Y$. The main emphasis is that the finite sets in question will be explicitly constructed and will reflect the metric geometry of the group.  We will focus on the metric geometry defined by a Cayley graph,  or a presentation of the group as a uniform lattice subgroup of the group of isometries of a hyperbolic metric space.  
  
Let us first formulate our main result in one important special case, as follows. 

\subsection{Convergence of information along geodesic segments in the free group} \label{sec:RFreeGroup}
Consider the free group $\Gamma = \FF_r$ on $2r$ ($r \geq 2$) generators, taken from the symmetric free generating set $F=\set{a_i, a_i^{-1}\,;\, 1 \le i\le r}$.

The boundary of the free group is defined as
\[
\partial\FF_r = \big\{ \xi=(\xi_0,\, \xi_0\xi_1, \xi_0\xi_1\xi_2,\, ... ) :\, \xi_i \in F,\, \xi_{i+1} \neq \xi_i^{-1} \mbox{ for all } i \geq 0  \big\}.
\]
The boundary can be identified with the space of one-sided infinite geodesics w.r.t.\@ the left-invariant word metric $d_F$ in $\FF_r$, connecting the origin $e$ to a boundary point $\xi$. For each boundary point $\xi$, we consider the following increasing sequence of finite subsets of group elements $(n \ge 1)$ :  
$$F_n(\xi)=\set{e, \xi_0, \xi_0\xi_1, \xi_0\xi_1\xi_2,\dots,  \xi_0\xi_1 \xi_2\cdots\xi_{n-1}}= \set{e} \cup \set{\xi_0\cdots \xi_j\,:\,\, 0\le j \le n-1}\,.$$ 
These subsets are increasing one-sided geodesic segments starting at $e$, namely the first $(n+1)$ points on the geodesic $\xi$ in the Cayley tree defined by the generating set $F$, and originating at $e$.
Denote the standard Markov probability measure (defined by uniform subdivision) on $\partial \FF_r$ by $\nu$. 

\medskip

We can now state the following result:

\begin{Theorem}\label{free-gps-SMB}

Let $\Gamma = \FF_r$ be the free group on $2r$ generators where $r \in \NN_{\geq 2}$. 
Let $(X,\mu)$ be an ergodic p.m.p.\@ action of $\FF_r$ on a standard Borel probability space. Let $\cP$ be a  countable partition $\cP$ of $X$ with $H(\cP)< \infty$, and 
let 
$\cP^{F_n(\xi)}=\cP\vee\xi_0\cP\vee \xi_0\xi_1\cP\vee\cdots \vee \xi_0\xi_1\cdots\xi_{n-1} \cP$ 
 ($n\ge 1$) 
be the refinement of the partition along the geodesic $\xi$. Denote by $ \mathcal{J}(\cQ)$ the information function associated with a partition $\cQ$. Then the following limit exists, 
for $\nu$-almost every geodesic $\xi$ and $\mu$-almost every $x\in X$ :
\[  \lim_{n \to \infty} \frac{1}{n+1} \mathcal{J}\left( \cP^{F_n(\xi)}\right)(x) 
\,.\]
The value of the limit is in fact $\nu\times \mu$-almost surely constant.
This constant is given by the conditional entropy 
\[
H\left(\cP \,|\, \bigvee_{i=1}^\infty \cP^{F^1_i(\xi)}\right) 
\]
where $F^1_i(\xi)=F_i(\xi)\setminus \set{e}$, $i \ge 1$.

Convergence holds also in the $L^1(\partial\FF_r\times X, \nu\times\mu)$-norm, and in particular, for $\nu$-almost every geodesic $\xi$ :
\[  \lim_{n \to \infty} \frac{1}{n+1} H\left( \cP^{F_n(\xi)}\right)=\lim_{n \to \infty}\int_{\partial \FF_r} \frac{1}{n+1} H\left( \cP^{F_n(\xi)}\right)d\nu(\xi)
.\]
Finally, the law of the geodesic $\xi^{+}=\cup_{n \ge 1}F_n(\xi)$ is equivalent to $\nu$. 
\end{Theorem}  

Note that the existence of the last limit is an immediate consequence of the subadditivity of the numerical sequence  $\int_{\partial \FF_r} H\left( \cP^{F_n(\xi)}\right)d\nu(\xi)$, and we denote it by $\mathfrak{h}_\mathcal{P}(\alpha, X)$. (The usage of $\alpha$ in the notation will become clear below.)

Consider now the (Cayley graph of the) free group with the left-invariant word metric $d_F$ with respect to the symmetric generating system $F$ as defined above. Denote by $S_n$ the sphere of group elements $g \in \FF_r$ with word metric distance from the identity being equal to $n$. Then for each $g \in S_n$, there is a unique representation $g=a_1a_2 \cdots a_n$ in reduced form consisting of exactly $n$ elements $a_i \in F$. 
Imposing a somewhat more restrictive integrability assumptions on the partitions in question, it is possible to establish a pointwise almost sure convergence result for the information functions, when averaged on spheres in the free group, as follows. 

\begin{Corollary}\label{free-gps-SMB-int}
	In the situation of Theorem~\ref{free-gps-SMB}, assume in addition that there is some $\varepsilon_0 > 0$ such that $\sum_{P \in \cP} \mu(P) (- \log \mu(P))^{1 + \varepsilon_0} < \infty$. Then, the limit 
	\[
	\mathfrak{h}_{\cP}(\alpha, X) = 
	\lim_{n \to \infty}  \frac{1}{|S_n(e)|} \sum\limits_{\substack{g \in S_n, a_i \in F:  \\ g = a_1a_2 \cdots a_n}} \frac{\mathcal{J}\Big( \cP \vee \bigvee_{j=1}^n a_1 \cdots a_j\cP \Big)(x)}{n+1}
	\]
	 exists for $\mu$-almost every $x \in X$, and equals the constant in Theorem \ref{free-gps-SMB}.  
\end{Corollary}

\subsection{Fundamental domains and group-valued cocycles}
Before turning to describe some generalizations of the foregoing result, we begin by recalling the following general construction, which will be used repeatedly below. Let $\cX$ be a locally compact second countable space, with $\chi$ an invariant Radon measure of full support. Let $\Gamma$ be a countable group, and $L$ an lcsc (locally compact second countable) group, which both act on $\cX$ by measurable bijections, preserving the measure class of $\chi$. We assume that these actions commute elementwise, 
namely $\ell \gamma z=\gamma \ell z$, for all $\gamma \in \Gamma$, $\ell \in L$ and $z\in \cX$. Further, assume that $\cF$ is a Borel measurable strict fundamental domain of finite measure for the action of $\Gamma$ on $\cX$, namely any two distinct translates $\gamma 
\cF $ and $\gamma^\prime \cF$ are disjoint and their union over $\Gamma$ cover $\cX$. In that case the space of orbits $\Gamma \setminus \cX=\overline{\cX}$ of $\Gamma$ in $\cX$ is a standard Borel measure space with a finite Borel measure $\bar{\chi}$, and $L$ acts on it by measurable bijections preserving the measure class of $\bar{\chi}$. The space of orbits is in a natural measurable bijection $\psi$ with the fundamental domain $\cF$, given by $z\in \cF \mapsto \psi(z)=\Gamma z=\bar{z}$.

For each $z\in \cF \subset \cX$ and $\ell\in L$ the element $\ell z$ belongs to some uniquely determined $\Gamma$-translate of the fundamental domain, namely $\ell z\in \gamma^{-1}\cF$, or 
equivalently $\ell z=\gamma^{-1} z^\prime$, with $z^\prime\in \cF$.  Define a function 
$\zeta_0: L\times  \cF \to  \Gamma$, by setting $\zeta_0(\ell, z)=\gamma. $
We use $\zeta_0$ to define a function $\zeta : L \times \overline{\cX}\to \Gamma$ by  
$\zeta_0(\ell, z)=\zeta(\ell, \bar{z})$. We note that $L$ acts on $\overline{\cX}$, via 
$\ell\bar{z}=\psi(\ell \psi^{-1}(\bar{z}) )$.

\begin{Lemma}\label{Fund-dom-cocycle}
$\zeta: L\times \overline{\cX}\to  \Gamma$ is a strict measurable cocycle, namely 
$$\zeta(\ell_1\ell_2, \bar{z})=\zeta(\ell_1, \ell_2 \bar{z})\zeta(\ell_2, \bar{z})$$
\end{Lemma}
\begin{proof}
Let $\bar{z}\in \overline{\cX}$, and let $z$ denote the unique point in $\cF$, such that $z\in \bar{z}=\Gamma z$, namely $z=\psi^{-1}(\bar{z})$. Now write $\ell_1\ell_2 z=\gamma^{-1} _{\ell_1\ell_2, z}z^\prime$, where $\gamma_{\ell_1\ell_2, z}\in \Gamma$ and $z^\prime \in \cF$ are uniquely determined. Furthermore write $\ell_2 z=\gamma_{\ell_2, z}^{-1} z''$, where $z''\in \cF$.  
Then $\ell_1\ell_2 z=\ell_1 \gamma_{\ell_2, z}^{-1} z''$ and since 
the actions of $\Gamma$ and $L$ commute, this expression equals 
$ \gamma^{-1}_{\ell_2, z}\ell_1 z''= \gamma^{-1}_{\ell_2, z}\gamma^{-1}_{\ell_1, z''}z'''$, with $z'''\in \cF$. Therefore 
$$\ell_1\ell_2 z=\gamma^{-1} _{\ell_1\ell_2, z}z^\prime= \gamma^{-1}_{\ell_2, z}\gamma^{-1}_{\ell_1, z''}z'''=\left(\gamma_{\ell_1, z''}\gamma_{\ell_2, z}\right)^{-1}z''',.$$
In the foregoing equation, $z,z'''\in \cF$ and the group elements translating them are in $\Gamma$, so by the uniqueness implied by the fact that $\cF$ is a strict fundamental domain
$$\gamma^{-1} _{\ell_1\ell_2, z}= \gamma^{-1}_{\ell_2, z}\gamma^{-1}_{\ell_1, z''}=\left(\gamma_{\ell_1, z''}\gamma_{\ell_2, z}\right)^{-1}\,.$$
Finally, note  that $\overline{z''}=\Gamma z''= \Gamma{\ell_2 z}=\ell_2\Gamma z=\ell_2\bar{z}$, so that since $\zeta(\ell, z)=\gamma_{\ell, z}$ we have 
$$\zeta(\ell_1\ell_2, \bar{z})=\zeta(\ell_1, \ell_2\bar{z})\zeta(\ell_2, \bar{z})$$
as stated.

\end{proof} 
  
\subsection{Fundamental groups of closed manifolds of negative curvature} \label{sec:RManifold-1}

\subsubsection{Construction of a geometric cocycle}  
We now recall the following well-known construction. 
Let $M$ be a closed Riemannian manifold of negative curvature, by which we mean that the sectional curvature of $M$ is bounded between two strictly negative constants. Let $\Gamma=\pi_1(M)$ denote the fundamental group of $M$, and let $\widetilde{M}$ denote its universal cover. $\widetilde{M}$ carries a Riemannian metric which we denote by $\tilde{d}$, and $\Gamma$ acts freely and isometrically with respect to this metric, so that $M$ is isometric to the space of orbits $\Gamma\setminus \widetilde{M}$. Fixing a point $o\in \widetilde{M}$, the map $\gamma\mapsto \gamma\cdot o$ is injective, the metric $\tilde{d}$ can be restricted to the $\Gamma$-orbit and gives rise to a left-invariant metric $d$ on $\Gamma$ via 
$d(\gamma_1, \gamma_2):=\tilde{d}(\gamma_1 \cdot o, \gamma_2\cdot o)$. 
The action of $\Gamma$ on $\widetilde{M}$ admits a compact geometric fundamental domain $\widetilde{D}$.  By this we mean that 
\begin{itemize}
\item $\widetilde{D}$ is a bounded open smooth connected and simply connected submanifold of $\widetilde{M}$, whose boundary consists of a finite union of submanifolds contained in embedded smooth hypersurfaces of $\widetilde{M}$. Furthermore these boundary components have the same properties when viewed as submanifolds of the embedded hypersurfaces. 
\item The union of the $\Gamma$-translates of the closure $\overline{\widetilde{D}}$ covers $\widetilde{M}$, namely $\cup_{\gamma\in \Gamma} \gamma \overline{\widetilde{D}}=\widetilde{M}$,
\item If $\gamma_1\neq \gamma_2$, then $\gamma_1 \widetilde{D} \cap \gamma_2 \widetilde{D}=\emptyset$. 
\end{itemize} 
 One way to construct a geometric fundamental domain is to define $\widetilde{D}$ to be the set of all points in $\widetilde{M} $ strictly closer to $o$ than to all the other points in the orbit $\Gamma\cdot o$. This construction yields the Dirichlet fundamental domain of $\Gamma$ in $\widetilde{M}$, which is indeed connected and simply connected, cf.\@ e.g.\@ \cite[\S~6.5]{Ra94}.
 We note that given a geometric fundamental domain it is  possible to choose a (piecewise smooth) Borel subset $\widetilde{D}_0\subset \partial \widetilde{D}$, such that $\widetilde{D}_1=\widetilde{D}\cup \widetilde{D}_0$ is a strict measurable fundamental domain, namely distinct $\Gamma$-translates of $\widetilde{D}_1$ are disjoint, and the union of the $\Gamma$-translates covers $\widetilde{M}$. We denote the projection of 
 $\widetilde{D}$ to $M$ under the canonical map $\widetilde{M}\to M$ by $D$, which is an open dense connected and simply connected submanifold of $M$. Similarly, $D_1$ denotes the projection of $\widetilde{D}_1$.

Let $S^1 \widetilde{M}$ denote the unit tangent bundles of $\widetilde{M}$, so that $\Gamma\setminus S^1 \widetilde{M}=S^1 M$ is the unit tangent bundle of $M$, and let $\tilde{g}_t : S^1 \widetilde{M}\to S^1 \widetilde{M}$ be the geodesic flow associated with the Riemannian metric. The flow on $S^1 \widetilde{M}$ descends to the geodesic flow $g_t : S^1 M \to S^1 M$ on the unit tangent bundle of $M$.  
Given a point $(p,v)\in S^1M$, where $p$ is a point in $M$ and $v$ is a unit tangent vector at $p$, the image of $(p,v)$ under the geodesic flow is obtained as follows. First take the unique lift of $(p,v)$ to a point $(\tilde{p}, v)\in S^1 \widetilde{M}$ with $ \tilde{p}\in \widetilde{D}_1$, then translate it by the transformation $\tilde{g}_t : (\tilde{p}, v) \mapsto (\tilde{p}_t, d(\tilde{g}_t)_\ast (v))=(\tilde{p}_t, v_t)$ where $\tilde{p}_t\in \widetilde{M}$ is the point on the geodesic starting at $\tilde{p}$ in the direction $v$ which is at (signed) distance $t$ from $\tilde{p}$, and finally project the point $(\tilde{p}_t, v_t)\in S^1 \widetilde{M}$ to $S^1 M$. 
The resulting point is denoted by $g_t(p,v)=(q_t, v_t)$. 

Fix an origin $o\in D\subset M$ with  $\tilde{o}\in \widetilde{D}\subset \widetilde{M}$. We now define a map $\beta : \RR\times S^1D_1 \to \pi_1(M, o)$ by the following procedure. Given $(p,v)\in S^1D_1 $, consider a concatenation of  three paths: first 
consider a smooth closed curve connecting $(\tilde{o}, v)$ and $(\tilde{p},v)$ in $S^1 \widetilde{D}_1$. Then connect $(\tilde{p},v)$ and $(\tilde{p}_t, v_t)$ by the path  described above, which under the map $S^1 \widetilde{M}\to \widetilde{M}$ covers the geodesic in $\widetilde{M}$ between $\tilde{p}$ and $\tilde{p_t}$. Finally, $\tilde{p}_t$ belongs to exactly one translate of the strict $\Gamma$-fundamental domain 
$\widetilde{D}_1$, which we denote by  $\gamma^{-1}  \widetilde{D}_1$ 
for a uniquely determined $\gamma\in \Gamma$, and we choose a smooth closed path in $\gamma^{-1} \widetilde{D}_1$ between $\tilde{p}_t$ and $\gamma^{-1} \tilde{o}$. 

 We now follow the discussion of \cite{FW97}, and define the following two maps, which will turn out to be essentially equal.  
 (We note however that in \cite{FW97} the fundamental domain is for the right action of $\Gamma$, whereas in our discussion below $\Gamma$ acts from the left.)
\begin{enumerate}
\item $a :\RR\times  S^1M\to \Gamma$, where $a(t,(p,v))$ is the unique $\gamma\in \Gamma$ such that the endpoint $\tilde{p_t}$ of the geodesic of length $\abs{t}$ described above between $\tilde{p}$ and $\tilde{p}_t$ belongs to $\gamma^{-1}\widetilde{D}_1$. 
Since $\widetilde{D}_1$ is a strict measurable fundamental domain for the action of $\Gamma$ on $\widetilde{M}$, it follows that $\gamma$ is well defined, and furthermore that $a$ is a strict Borel-measurable cocycle : $a(s+t,(p,v))=a(s, (p_t,v_t))a(t,(p,v))$, for all $(p,v)\in S^1D_1$ and $t,s\in \RR$. The cocycle $a$ depends on the choice of origin $o\in M$ and also on the choice of the measurable fundamental domain $D_1$. Using the identification of $D_1$ with $M$, $a$ can also be viewed as a cocycle $ a:  \RR\times S^1D_1\to \Gamma$. 
\item 
The projection of the concatenation of the three paths described above to $M$ is a closed path. Assume further that $\tilde{p}$ and $\tilde{p}_t$ actually belong to $\widetilde{D}$ and to $\gamma^{-1}\widetilde{D}$. Then the first path and the third path can be taken to lie entirely in $\widetilde{D}$ (and otherwise their endpoint is in $\widetilde{D}_1\setminus \widetilde{D}$). Then the homotopy class of the resulting closed path in $M$ described above is independent of the choice of the first and the third paths, due to the fact that $\widetilde{D}$ is simply connected. Therefore define $b(t, (p,v))=\gamma$, where $\gamma\in \pi_1(M, o)$ represents the homotopy class in question.

\end{enumerate}

Recall that $S^1M$ carries an ergodic probability measure invariant under the geodesic flow, called the Bowen-Margulis measure, which is also the measure of maximal Kolmogov-Sinai entropy for the geodesic flow on the unit tangent bundle. Under the correspondence between geodesic-flow-invariant probability measures on $S^1M$ and $\Gamma$-invariant probability measures on the double boundary $\partial^2 \widetilde{M}= (\partial\widetilde{M} \times \partial\widetilde{M}) \setminus \Delta \partial \widetilde{M}$ of the universal cover $\widetilde{M}$ (see \cite{Ka90} for a full discussion) the Bowen-Margulis measure corresponds to the Patterson-Sullivan measure. We denote the Bowen-Margulis probability measure by $\mu^{BM}$, and note the following.  
\begin{Claim}\label{homotopy}  (see \cite[Prop. 5]{FW97}). 
\begin{enumerate}
\item Given any $t,s \in \RR$, there exists a conull set in $S^1M$ w.r.t. the measure $\mu^{BM}$, such that for $(p,v)$ belonging to this set we have $a(t,(p,v))=b(t,(p,v))$, $a(s,(p,v))=b(s,(p,v))$, and $a(s+t,(p,v))=b(s+t, (p,v))$. 
\item Let $\gamma$ denote $a(t,(p,v))$, where $p_t\in \gamma^{-1}\widetilde{D}$. Then the three-part path described above descends to a closed path, which is homotopic to the projection of the geodesic between $\tilde{o}$ and $\gamma^{-1} \tilde{o}$, namely to the closed path represented by $\gamma$. 
Thus in this case $a (t,(p,v))=b(t, (p,v))$. 
\end{enumerate}
\end{Claim}

The length of the geodesic path in $\widetilde{M}$ connecting $\tilde{p}$ and $\tilde{p}_t$ is $\abs{t}$, and since $\widetilde{D}_1$ is bounded, and so are the distances 
$\tilde{d}(\tilde{o},\tilde{p})$ and $\tilde{d}(\gamma^{-1} \tilde{o},\tilde{p_t})$.  
Thus the three-part concatenated path between $\tilde{o}$ and $\gamma^{-1} \tilde{o}$ is a $C$-almost-geodesic for some fixed constant independent of $p\in M$ and $t\in \RR$, namely it satisfies 
$$\abs{d(a(t_1, (p,v))^{-1}, a(t_2, (p,v))^{-1})-\abs{t_1-t_2}}\le C\,.$$
Indeed, our definition above of $a(t, (p,v))$ is such that 
$\tilde{p}_t \in a(t, (p,v))^{-1} \widetilde{D}_1$. Consider for $(p,v)\in S^1M$  the map $(t,(p,v))\to a(t,(p,v))^{-1}\in \Gamma$. The image of this map is a $C$-almost geodesic in $\Gamma$ with respect to the metric $d$ (defined by restriction of $\tilde{d}$ to $\Gamma\cdot o$),  as follows from the fact that $\Gamma$ has a bounded fundamental domain in $\widetilde{M}$, which also implies that  $C$ is uniform over $\RR$ and $S^1M$. Furthermore, for any $t_0\neq 0$ if we consider the map $(nt_0,(p,v))\mapsto b(nt_0,(p,v))^{-1}$ where $p\in D$, then there exists a conull subset of $(p,v)\in S^1 \widetilde{D}$ where this map is a well-defined $C$-almost geodesic, and 
$b(nt_0,(p,v))=a(nt_0,(p,v))$. 

\subsubsection{Pointwise convergence of information along almost-geodesic segments}

Let us now fix $t_0\in \RR\setminus {0}$ and the corresponding isometry $g_{t_0}$. It acts as a probability-measure-preserving map on $S^1(M)=Y$, which we denote by $S$, suppressing the dependence on $t_0$. We consider the orbit relation generated in $Y$ by the powers of $S$, and define 
$\alpha(y, S^n y)=a(nt_0, (p,v))^{-1}$, with $y=(p,v)$. 

Let us consider the following variable finite subsets of $\Gamma $, with $t_0 >  0$ fixed
\begin{equation}\label{quas-geod} F_n(p,v) =F_n(y)=\set{\alpha(y, S^k y)
\,;\,  0\le k \le n}\subset \Gamma\,\,\,\,,\,\,\,\,F_n^1(p,v)=F_n(p,v)\setminus \set{e}\,\,.\end{equation}
Thus $F_n(p,v)$ arises from starting from the point $\tilde{p}\in \widetilde{D}$ and following the geodesic emanating from it in the direction $v$, and recording the fundamental domains where the geodesic is located at times $0, t_0, 2t_0, \dots, nt_0$ 
by $a(kt_0, (p,v))^{-1}\widetilde{D}_1$.

As noted above for $\mu^{BM}$-almost every $(p,v)\in S^1 M$ these subsets
 are {\it $C$-almost-geodesic segments} in $\Gamma$.  We will suppress the choice of $t_0$ from the notation, but note that $C$ depends on $t_0$.   The union 
$$F(p,v)=\bigcup_{n\in \NN} F_n(p,v)=\set{a(kt_0, (p,v))^{-1}\,;\,  k \ge 0}$$ 
is a one-sided $C$-almost geodesic ray, namely 
$$\abs{d(a(k_1t_0, (p,v))^{-1}, a(k_2t_0, (p,v))^{-1})-\abs{k_1t_0-k_2t_0}}\le C\,.$$

Let $(X,\mu)$ be an ergodic  p.m.p. action of $\Gamma$ on standard Borel probability space. Let $\cP$ be a  countable partition $\cP$ of $X$ with $H(\cP)< \infty$, and 
consider the partitions 
$$\cP^{F_n(p,v)}=\bigvee_{\gamma\in F_n(p,v)} \gamma \cP$$
arising as the refinement of the partition $\cP$ along the $C$-almost-geodesic segments defined by $(p,v)$. 

We define, using subadditivity  : 
\[\lim_{n \to \infty}\int_{S^1 M} \frac{1}{n+1} H\left( \cP^{F_n(p,v)}\right)d\mu^{BM}(p,v)= \mathfrak{h}_\mathcal{P}(\alpha, X)\,.\]

Recall that we denote by $ \mathcal{J}(\cQ)$ the information function associated with a partition $\cQ$. 
We can now state the Shannon-McMillan-Breiman theorem along the variable finite sets given by the $C$-almost geodesic segments just defined in negatively curved groups. 

\begin{Theorem}\label{neg-curv--SMB}
Keeping the notation introduced above, let $\Gamma=\pi_1(M)$ be the fundamental group of a closed Riemannian manifold of strictly negative curvature, $(X,\mu)$ be an ergodic p.m.p.\@ action of $\Gamma$, and $\cP$ be a  countable partition $\cP$ of $X$ with $H(\cP)< \infty$. 
Then for $\mu^{BM}$-almost every $(p,v)$, 
the following limit exists 
for $\mu$-almost every $x\in X$ :
\[  \lim_{n \to \infty} \frac{1}{n+1} \mathcal{J}\left( \cP^{F_n(p,v)}\right)(x)\,.\]
The value of the limit is $\mu^{BM}\times \mu$-almost surely the constant 
$\mathfrak{h}_\mathcal{P}(\alpha, X)$. It is also given by the conditional entropy 
\[
H\left(\cP \,|\, \bigvee_{i=1}^\infty \cP^{F^1_i(p,v)}\right)=\mathfrak{h}_\mathcal{P}(\alpha, X)\,,
\]
where $F^1_i((p,v))=F_i((p,v))\setminus \set{e}$.

Convergence holds also in the $L^1(S^1 M\times X, \mu^{BM}\times\mu)$-norm, and in particular, for $\mu^{BM}$-almost  $(p,v)$ the partitions refined along the $C$-almost-geodesic $F(p,v)$ satisfy 
\[  \lim_{n \to \infty} \frac{1}{n+1} H\left( \cP^{F_n(p,v)}\right)
= \mathfrak{h}_\mathcal{P}(\alpha, X)\,.\]
Finally, the law of $F^+(p,v)$  belongs to the Patterson-Sullivan measure class on the boundary $\partial \widetilde{M}$. 
\end{Theorem} 


\subsection{Locally symmetric spaces of finite volume} \label{sec:Rlocallysymmetricspaces}
Let $G$ be a connected non-compact almost simple Lie group of real-rank one. 
Let $\theta$ be a Cartan involution on $G$ and its Lie algebra $\mathfrak{g}$, and let $K$ be the maximal compact subgroup of $G$ consisting of fixed points of the Cartan involution. Let $A=\set{a_t\,;\, t \in \RR}$ be the associated $\RR$-split torus, and $N$ the nilpotent group whose Lie algebra $\mathfrak{n}$ is the sum of the positive roots of $A$ in $\mathfrak{g}$. Then $G=KAN$ is an Iwasawa decomposition of $G$, 
and $G=KA_{t\ge 0}K$ is the geodesic polar coordinate decomposition. Denote by $K_0$ the centralizer of $A$ in $K$, and by $P$ the minimal parabolic subgroup $P=K_0 AN$. Then $G/P=K/K_0$ is the maximal boundary of $G$. If $\cS$ is the symmetric space associated with $G$, then $G$ acts ismetrically and transitively on $\cS$, and each maximal compact subgroup has a unique fixed point. Denoting by $\tilde{o}$ the point left fixed by $K$, we have the identification $\cS=G/K$. Every geodesic in $\cS$ passing through $\tilde{o}$ is an orbit of the group $A$, and the set of unit tangent vectors to these geodesics at $\tilde{o}$ is parametrized by $K/K_0$. Thus the boundary of the symmetric space $\cS$ can be identified with $K/K_0=G/P$. The space of ordered pairs of distinct boundary points, which is  given by $\left(G/P\times G/P\right)\setminus \Delta(P)$ (where $\Delta(P)$ is the diagonal), therefore parametrizes the set of bi-infinite geodesics in $\cS$ passing through $\tilde{o}$. 
This space can also be identified with the homogeneous space  $G/(AK_0)$ of cosets of $AK_0$ in $G$, under the correspondence assigning to $gAK_0$ the image under $g$ of the two distinct limit boundary points of the geodesic $A\cdot \tilde{o}$. This correspondence is clearly $G$-equivariant so the two homogeneous spaces $\left(G/P\times G/P\right)\setminus \Delta(P)$ and $G/(AK_0)$ are $G$-isomorphic. 

Let now $\Gamma$ be any lattice subgroup of $G$, and consider the homogeneous space $\Gamma\setminus G$, on which we take the normalized $G$-invariant probability measure, denoted $m_{\Gamma\setminus G}$. The symmetric space $\cS=G/K$ is a simply connected space of strictly (variable) negative sectional curvature bounded between two constants, which is identified with real, complex, quaternionic or octonionic hyperbolic space. 

When $\Gamma$ is a torsion-free uniform lattice, then $M=\Gamma\setminus G /K$ is a closed hyperbolic manifold (of one of these types), and $\cS=\widetilde{M}$ is its universal cover. The homogeneous space $\Gamma \setminus G/K_0$ is a bundle over $M$ with typical fiber $K/K_0$, and coincides with the unit tangent bundle $S^1M$ of $M$, since $K/K_0$ can be identified with the sphere of tangent directions to the point $\tilde{o}\in \cS$ and also with the boundary sphere at infinity. Since $K_0$ centralizes the group $A$, it follows that $A$ acts on the homogeneous space $S^1M$ via $(a_t,\Gamma g K_0)\mapsto \Gamma ga^{-1}_t K_0$. The one-parameter flow defined by $A$ is none other that the geodesic flow on the unit tangent bundle of the locally symmetric closed real manifold $M$ (as noted already by Gelfand and Fomin in the 1950's). This flow coincides with the flow discussed in the previous subsection, when the manifold in question is locally symmetric and closed. We also note that  the case of a closed locally symmetric space $\Gamma\setminus G /K_0$ and its unit tangent bundle, 
the Riemannian volume  coincides with Bowen-Margulis measure.

Any lattice $\Gamma$ in $G$ is finitely generated, and we can choose a Dirichlet finite-sided fundamental domain $D$ for it in $\cS$. We can lift $D$ to  $\widetilde{D}\subset \cS$ and use it to define a measurable strict fundamental domain $\widetilde{\cD}_1$ for $\Gamma$ in $G$. Thus $\widetilde{\cD}_1 \subset G$ is such that $\gamma_1\widetilde{\cD}_1\cap \gamma_2 \widetilde{\cD}_1=\emptyset$ when $\gamma_1\neq \gamma_2$, and $\cup_{\gamma \in \Gamma} \gamma\widetilde{ \cD}_1=G$. Then the Haar measure of $\widetilde{\cD}_1$ is the volume of $\Gamma\setminus G$, namely $1$ under our assumptions. When $\Gamma $ is co-compact, we can assume that $\widetilde{\cD}_1$ is bounded, but this is no longer the case when $\Gamma$ is non-uniform. However, in the present set-up, for any lattice we can consider still consider the fundamental 
domain $\widetilde{D}$ for the action of $\Gamma$ on $G/K_0=S^1\cS=S^1\widetilde{M}$, and a corresponding strict measurable fundamental domain $\widetilde{D}_1$. We identify $\widetilde{\cD}_1$ with $\Gamma \setminus G$. 
We define the cocycle $s : G \times \Gamma \setminus G \to \Gamma$ given by 
$s (g, \Gamma y)=\sigma(\Gamma yg^{-1})g \sigma(\Gamma y)^{-1} $, where $\sigma: \Gamma \setminus G \to \widetilde{\cD}_1\subset G$ is a measurable section, bounded on compact sets, with $\sigma(\Gamma e)=e$. This cocycle satisfies that $xg^{-1}\in s(g,x)^{-1} \widetilde{\cD}_1$ for any $x \in \widetilde{\cD}_1$.

The flow $a_t : \Gamma \setminus G \to \Gamma \setminus G$ given by $(a_t, \Gamma g)\mapsto \Gamma g a^{-1}_t$ preserves the $G$-invariant probability measure $m_{\Gamma\setminus G}$, and is ergodic by the Howe-Moore theorem (or more directly by the Gelfand-Fomin argument).  We consider the well-defined  quotient flow   $a_t : \Gamma\setminus G /K_0 \to \Gamma\setminus G /K_0$ on the orbifold 
$Y=\Gamma \setminus G /K_0$, given by $(a_t,\Gamma g K_0)\mapsto \Gamma g a^{-1}_t K_0$. We denote the quotient measure on $Y$ by $\nu$, which is of course also an ergodic probability measure under the flow.  When the lattice is torsion-free and uniform, this construction specializes to the geodesic flow on the unit tangent bundle discussed above.  

Now define the finite sets in the lattice $\Gamma$, parametrized by $\Gamma g K_0=y\in Y=  \Gamma\setminus G /K_0$ (and depending on a choice of a fixed $t_0\neq 0$), using the cocycle $s : G\times \Gamma\setminus G\to \Gamma$ restricted to the subgroup $\set{a_t\,;\, t\in \RR}\subset G$, as follows :
\begin{equation}\label{cocycle-values} F_n(y) =\set{s(a_{jt_0}, y)^{-1}\,;\,  0\le j \le n}\subset \Gamma\,\,\,\,,\,\,\,\,F_n^1(y)=F_n(y)\setminus \set{e}\,\,.\end{equation} 
Let us also denote $F(y)=\cup_{n\ge 0} F_n(y)$.  
 We can now formulate the following result. 
 
\begin{Theorem}\label{lattice-sbgs} (Shannon-McMillan-Breiman theorem for lattice subgroups.)
Let $\Gamma$ be any lattice in an almost simple connected Lie group of finite center and real rank one. 
Let $(X,\mu)$ be a p.m.p.\@ ergodic action of  $\Gamma$, and $\cP$ a countable partition of finite Shannon entropy. 
Then for $\nu$-almost every $y \in S^1(M)$, 
the following limit exists 
for $\mu$-almost every $x\in X$ and the limit is $\nu \times \mu$-almost surely constant :
\[  \lim_{n \to \infty} \frac{1}{n+1} \mathcal{J}\left( \cP^{F_n(y)}\right)(x)\,.\]

\end{Theorem}

\subsection{Pointwise convergence of information  along  almost-geodesics in hyperbolic groups }\label{sec:hyp}
Let $H$ denote a proper and quasi-convex Gromov-hyperbolic metric space, and $\Gamma$ a group of isometries acting properly and cocompactly on $H$. We denote by $\partial H$ the Gromov boundary of $H$, and $\nu_{PS}$ the Patterson-Sullivan measure on $\partial H$ associated with some fixed origin point $o\in H$. For a general exposition of the fundamental facts about both we refer to \cite{GdH88}, \cite{Co93} and \cite{Ca11}. 
In this generality, Bader and Furman (\cite{BF17}, we will use the revised version \cite{BF22}) constructed and gave a detailed description of a measurable p.m.p.\@ flow $g_t$ on a probability space $(Y,\nu)$, which generalizes the classical geodesic flow on closed manifolds of negative curvature discussed above. In particular, the construction includes a measurable cocycle  $a : \RR\times Y\to \Gamma$, whose (inverse) values produce almost surely a $C$-almost-geodesic with $C$ uniformly bounded. Furthermore, the cocycle is shown to satisfy the key property of weak-mixing, see Definition \ref{weak-mix} below and \cite[\S 6]{BF22}. 
The natural generalizations of the results stated in Theorem \ref{neg-curv--SMB} apply to this much more general set-up. 
We now summarize the results from \cite{BF22} that we will utilize below in more detail, as follows.   
 
\begin{enumerate}
\item The diagonal action of $\Gamma$ on $\partial^2 H=(\partial H\times \partial H)\setminus \Delta(\partial H)$ is ergodic with respect to $\nu_{PS}\times \nu_{PS}$ (\cite[Thm.\@ 1.4]{BF22}), and hence so  is its equivariant factor action on $(\partial H,\nu_{PS})$. Furthermore, both actions are essentially free \cite[Thm.\@ 1.9]{BF22}. Denote by $r(\gamma,\xi)=\frac{d\gamma^{-1}\nu_{PS}}{d\nu_{PS}}(\xi)$ the Radon-Nikodym derivative of $\nu_{PS}$. 
\item  There exists a Radon measure ${\bf m}=m^{BMS}$ on $\partial^2 H$ equivalent to $\nu_{PS}\times \nu_{PS}$, which is invariant (and infinite) under the $\Gamma$-action.  
It is constructed 
via a positive measurable function $\exp (F (\xi,\xi^\prime))$ on $\partial^2 H$ satisfying 
$$r(\gamma,\xi)r(\gamma,\xi^\prime)=\frac{\exp (F(\gamma\xi,\gamma\xi^\prime))}{\exp (F(\xi,\xi^\prime))}\,,$$
so that $d{\bf m}=\exp (F(\xi,\xi^\prime))d\nu_{PS}(\xi)d\nu_{PS}(\xi^\prime)$ is a $\Gamma$-invariant measure (see \cite[Prop. 4.3 \& \S 4.C]{BF22}). 
\item If $L$ denotes the Lebesgue measure on $\RR$, there is an action of $\Gamma$ on $\partial^2 H\times \RR$,  preserving the infinite measure $\nu_{PS}\times\nu_{PS}\times L$, 
given by 
\begin{equation}\label{extension} 
\gamma\cdot (\xi,\xi^\prime,t)=\left(\gamma \xi, \gamma\xi^\prime, t+\tau(\gamma, \xi,\xi^\prime)\right )\text{\,\,,\,\,where \,\,} \tau(\gamma, \xi,\xi^\prime)=
\frac{\log r(\gamma,\xi)-\log r(\gamma,\xi^\prime)}{2\delta_\Gamma}
\end{equation} 
 with $\delta_\Gamma$ the growth exponent of the metric on $\Gamma$  (\cite[eq. (1.1)]{BF22}), see \cite[\S\S~4C, 4D \& Prop. 4.7]{BF22}.  
\item 
There exists a measurable map $\pi : \partial^2 H\times \RR\to H$ such that 
$\pi(\xi,\xi^\prime, \cdot) : \RR\to H$ is $C$-almost geodesic with limit points $\xi$ and $ \xi^\prime$, and $\pi(\xi,\xi^\prime, 0)$ belongs to a compact subset of $H$ containing the origin $o\in H$ of uniformly bounded diameter (\cite[Prop. 3.4]{BF22}). 
\item  The action of $\Gamma$ on  $\partial^2 H\times \RR$ 
  is measure-theoretically smooth and has a finite-measure pre-compact fundamental  domain, denoted $\cD_\Gamma=\cD$   
(\cite[Thm.\@ 1.13]{BF22}). Normalizing the measure of the fundamental domain $\cD$ for the action of $\Gamma$ to be $1$,
an invariant Radon measure on $\partial^2 H$ equivalent to $\nu_{PS}\times \nu_{PS}$ is determined uniquely and denoted 
${\bf m}$.  
\item The action of $\Gamma$ on $(\partial^2 H,\nu_{PS}\times \nu_{PS})$ is weak-mixing (\cite[Cor. 1.10]{BF22}, see Definition \ref{weak-mix} below), and hence the action of $\Gamma$ on $(\partial H, \nu_{PS})$ is weak-mixing as well. 
\item  The measure-preserving flow $g_t : (\xi,\xi^\prime, s)\to  (\xi,\xi^\prime, s+t)$ on the extension $(\partial^2 H\times \RR, {\bf m}\times L)$ commutes with the $\Gamma$ action and so descends to the space of $\Gamma$-orbits in the extension. This space is denoted $\Gamma\setminus \left(\partial^2 H\times \RR\right)=S^1 M_\Gamma=S^1 M$, and the resulting flow on it preserves a probability measure denoted $\mu^{BM}$.  
The probability measure $\mu^{BM}$ on $S^1 M$ arises from its identification with a pre-compact fundamental domain $\cD_\Gamma=\cD\subset \partial^2 H\times \RR$ of unit measure. 
\item The p.m.p. flow $g_t$ is called the geodesic flow, and it is ergodic \cite[Cor. 1.6]{BF22}, but may fail to be weak-mixing (see the discussion following Thm.\@ 1.13 in \cite{BF22}).
\item  The $\Gamma$-fundamental domain $\cD\subset \partial^2H \times \RR$ gives rise to a measurable map $a_\cD : \RR\times \cD\to  \Gamma$ via 
$$a_\cD(g_t, (\xi,\xi^\prime,s))=\gamma \text{\,\,\,\,       if and only if     \,\,\,\,} g_t (\xi,\xi^\prime, s) \in \gamma^{-1}\cD\,.$$ 
Under the identification of $\cD$ and the space $S^1 M$ of $\Gamma$-orbits in $\partial^2H \times \RR$, $a_\cD$ gives rise to a cocycle for the $\RR$-action on $(S^1 M,\mu^{BM})$ given by the geodesic flow (see \cite[Lem.~4.10 \& Rem. 5.2]{BF22}).  
\item Denote the cocycle values by $a_\cD(t,(\xi,\xi^\prime,s))=\gamma_{t, (\xi,\xi^\prime,s)}$, for $\bf{m}$-almost every $(\xi,\xi^\prime,s)\in \cD$, and $t\in \RR$. Then the distances
$$d_H\left(\gamma^{-1}_{t, (\xi,\xi^\prime,s)}\cdot o, \pi(\xi,\xi^\prime,t)\right)$$
are globally bounded (\cite[Lem. 4.10]{BF22}). 
\item Thus $  \gamma^{-1}_{t, (\xi,\xi^\prime,s)} $ translates $o$ a distance $t+O(1)$ along an almost geodesic line 
$\pi(\xi,\xi^\prime,t)$ connecting $\xi^\prime$ and $\xi$. Equivalently, the cocycle values 
$\set{a_\cD(g_t,(\xi,\xi^\prime,s))^{-1}\,;\, t \in \RR}$ constitute a $C$-almost geodesic connecting $\xi^\prime$ and $\xi$. In particular $\abs{\abs{\gamma^{-1}_{t,(\xi,\xi^\prime,s)}}-\abs{t}}\le C$ (for a suitable fixed $C$), and 
$$\lim_{t\to \infty}\gamma^{-1}_{t,(\xi,\xi^\prime,s)}\cdot  o=\xi^\prime\,\,\,,\,\,\, \lim_{t\to -\infty}\gamma^{-1}_{t,(\xi,\xi^\prime,s)}\cdot  o=\xi$$
for almost every $(\xi, \xi^\prime, s)\in \cD$ (see \cite[Lem. 4.10]{BF22}). 
\item The fact that $\abs{\gamma^{-1}_{t,(\xi,\xi^\prime,s)}}=\abs{t}+O(1)$ implies that the geodesic flow is essentially free : the  stability group of almost every point is finite cyclic subgroup of $\RR$ and hence trivial. Furthermore, the  cocycle $a_\cD: \RR \times S^1 M \to \Gamma$ is a weak-mixing cocycle (see \cite[Cor. 1.16]{BF22}, see Def. \ref{weak-mix} below).

\end{enumerate}

As noted in item (8) above, the geodesic flow in the present general context is always ergodic but not always strong mixing, and in fact it may have a non-trivial Kronecker factor and fail to be weak-mixing. In general, whenever a p.m.p. $\RR$-action is strong mixing, clearly for any $t\neq 0$ the action restricted to the subgroup $\ZZ\cdot t$ is ergodic. Otherwise, there may be non-zero exceptional elements $0\neq s\in \RR$ where $\ZZ\cdot s$ is not ergodic. Nevertheless, the set $\cE$ of such exceptional elements is a null set, under the assumption that the original $\RR$-action is ergodic. 

 Let $(X,\mu)$ be an ergodic p.m.p. action of a hyperbolic group $\Gamma$, and consider the flow $g_t^X$ on the skew extension $(S^1 M\times X, \mu^{BM}\times \mu)$. By item (12), the fact that the geodesic flow is weak-mixing implies that the extened flow is ergodic, and we then fix a non-exceptional element $t_0\in \RR\setminus \cE$ so that $\ZZ\cdot t_0$ is also ergodic. Let us denote the isometry $\gamma_{t_0}$ by $S$. Define $\alpha(y, S^n y)=a_\cD(\xi,\xi^\prime, nt_0)^{-1}$  for each $y=(\xi,\xi^\prime, nt_0)$ ($y\in S^1M$), and define the measurable family of finite sets: 
 
$$F_n(y)=\set{\alpha(y, S^k y)\,;\, 0 \le k \le  n}, \,\,\,\,\,\,\, F^1_n(y)=F_n(y)\setminus \{e\}\,. $$
By item (11) for almost every $y$ the finite sequence just defined is a $C$-almost geodesic, and we consider the (one sided) $C$-almost geodesic $F^+(y)=\cup_{n \ge 0} F_n(y)$. Again, we suppress $t_0$ from the notation, but note that the value of $C$ depends on it. 

We define, as above, 
\[\lim_{n \to \infty}\int_{S^1 M} \frac{1}{n+1} H\left( \cP^{F_n(y)}\right)d\mu^{BM}(y)= \mathfrak{h}_\mathcal{P}(\alpha, X)\,.\]

We can now formulate : 

\begin{Theorem}\label{SMB-hyp-gps} 
Keeping the notation introduced above, let $(X,\mu)$ be an ergodic p.m.p. action of a hyperbolic group $\Gamma$, and $\cP$ be a  countable partition $\cP$ of $X$ with $H(\cP)< \infty$. 
Then for $\mu^{\text{BM}}$-almost every $y\in S^1 M$ and $\mu$-almost every $x\in X$, the following limit exists along the $C$-almost-geodesic segments $F^+(y)$:
\[  \lim_{n \to \infty} \frac{1}{n+1} \mathcal{J}\left( \cP^{F_n(y)}\right)(x).\]
The value of the limit is $\mu^{\text{BM}}\times \mu$-almost surely the constant $\mathfrak{h}_\mathcal{P}(\alpha, X)$. It is also given by the conditional entropy 
\[
H\left(\cP \,|\, \bigvee_{i=1}^\infty P^{F^1_i(y)}\right)=\mathfrak{h}_\mathcal{P}(\alpha, X)\,.
\]

Convergence holds also in the $L^1(S^1 M\times X, \mu^{\text{BM}}\times\mu)$-norm, and in particular, for $\mu^{\text{BM}}$-almost  $y\in S^1 M$ 
\[  \lim_{n \to \infty} \frac{1}{n+1} H\left( \cP^{F_n(y)}\right)
= \mathfrak{h}_\mathcal{P}(\alpha, X)\,.\]
Finally, the law of $F^+(y)$ belongs to the Patterson-Sullivan measure class on $\partial H$. 
\end{Theorem}

\subsection{ Rokhlin entropy, finitary entropy and orbital entropy}
Theorems \ref{free-gps-SMB}, \ref{neg-curv--SMB}, \ref{lattice-sbgs} and \ref{SMB-hyp-gps} raise two fundamental questions. 
\begin{itemize}
\item Defining the expression $\mathfrak{h}_\mathcal{P}(\alpha, X)$ requires extensive machinery, including  the auxiliary p.m.p. dynamical system $(Y, \nu, S)$, the cocycle $\alpha$, and the finite sets $F_n(y)$. Is there an underlying fundamental isomorphism-invariant of the dynamical system $(X,\mu,\Gamma)$ that is independent of these choices?  
\item Is the almost-sure limit expression in the Shannon-MacMillan-Breiman theorem along the sets $\set{F_n(y)\,;\, y\in Y}\subset \text{Fin }(\Gamma)$, namely 
$$ \lim_{n \to \infty} \frac{1}{n+1} \mathcal{J}\left( \cP^{F_n(y)}\right)(x)=\mathfrak{h}_\mathcal{P}(\alpha, X)=H\left(\cP \,|\, \bigvee_{i=1}^\infty \cP^{F^1_i(y)}\right)\, ,$$
related to other notions of entropy for actions of (non-amenable) groups?

\end{itemize}
Both of these questions have definitive and remarkable answers, which rely on a deep result of B. Seward.  First, recall that Seward defines:

\begin{Definition}[Rokhlin entropy]
Let $\Gamma \curvearrowright (X,\mu)$ be a p.m.p.\@ ergodic group action. Then, the number
\[
h^{\operatorname{Rok}}(\Gamma \curvearrowright X) := \inf\big\{ H(\cP)\,|\, 
\cP \mbox{ countable, generating partition}\big\}
\]
is called the {\em Rokhlin entropy} of the group action. 
\end{Definition}
The notion of Rokhlin entropy has its origin in Rokhlin's study of entropy for a single transformation, and Seward's definition above constitutes a far-reaching generalization of the classical case.  This notion of entropy was studied intensively
in the context of general measure-preserving group actions by Seward in the series of papers \cite{Se15a, Se15b, Se16}. 

Seward and Tucker-Drob showed in \cite{ST12}
that for free ergodic actions of amenable groups, Rokhlin entropy coincides
with the classical Kolmogorov-Sinai entropy, generalizing Rokhlin's classical result for a single p.m.p. transformation. 
In an important breakthrough  Seward \cite{Se16} established the following remarkable upper bound for the Rokhlin entropy for every countable, generating partition $\cP$, assuming the $\Gamma$ action on $X$ is ergodic and essentially free. 
\vskip0.1in
\noindent{\bf Seward's inequality.} Assume that $\Gamma \curvearrowright (X,\mu)$ is ergodic and free. Then
\begin{eqnarray}\label{thm:seward}
h^{\operatorname{Rok}}\big( \Gamma \curvearrowright X \big) \leq 
\inf_{F \in \text{Fin}(\Gamma)} \frac{1}{|F|}\, H\Big( \bigvee_{\gamma \in F} \gamma \cP \Big).
\end{eqnarray}
(In fact, Theorem~1.5 in \cite{Se16} establishes a stronger statement
involving entropy conditioned on $\Gamma$-invariant $\sigma$-algebras.)

By the standard subadditivity property of Shannon entropy,  
the right hand side of (\ref{thm:seward}) is bounded from above by $H(\cP)$. 
Hence, by passing to the infimum over all generating partitions, we obtain 
in fact equality of the lower and the upper bound.  Therefore, as noted in \cite{NP17}, it is natural to define the {\it finitary entropy} for ergodic actions of $\Gamma$ on $(X,\mu)$.

\begin{Definition}[Finitary entropy]
\[
h^{\operatorname{fin}}\big( \Gamma \curvearrowright X \big)
:= \inf\Big\{ \inf_{F \in \text{Fin}(\Gamma)} H\Big( \bigvee_{\gamma \in F} \gamma\cP \Big) /|F| \, \,\Big|\, 
\cP \mbox{ countable, generating partition}\Big\}.
\]
\end{Definition}
In the set-up of free ergodic actions, the finitary entropy clearly satisfies, using Seward's inequality (\ref{thm:seward}),  that  
 $h^{\operatorname{fin}}\big( \Gamma \curvearrowright X \big)=h^{\operatorname{Rok}}(\Gamma \curvearrowright X)$. 
This consequence of Seward's inequality is crucial to our approach, which is based on constructing a collection of {\it variable finite subsets} $\cF_n(y)_{y\in Y}\in \text{Fin}(\Gamma)$, for some auxiliary dynamical system defined on $Y$. Indeed, suppose that some collection of variable finite subsets constructed according to some method has the property that the 
information functions of a partition of $X$, when refined along these finite sets, actually converge to a constant limit. Then by Seward's inequality that 
limit is bounded from below by the Rokhlin entropy of the space,  and when we take the infimum over generating partition, equality will hold between the infimum of these limits and the Rokhlin entropy. 

In \cite{NP17} we considered one approach to construct such sets, using finite sets 
$\cR_n(y)$ which are equivalence classes of a equivalence relation $\cR_n$, where the sequence of these relations constitutes a hyper-finite exhaustion of a suitable p.m.p. amenable ergodic equivalence relation, which admits a suitable cocycle into $\Gamma$. 
In the present paper we focus on the amenable ergodic equivalence relation generated by the orbits of a single p.m.p. ergodic transformation $ S : Y \to Y$, when the dynamical system is equipped with a suitable  $\Gamma$-valued  cocycle $\alpha$ defined on the orbit relation generated by $S$. 
We can now define 
\begin{Definition}\label{relative-entropy-def}

Define the {\em orbital entropy of the partition $\cP$ of $X$ w.r.t. to the cocycle $\alpha$}  by: 
\[
	\mathfrak{h}_\cP(\alpha,X) = \lim_{n \to \infty} \frac{1}{n+1}\int_Y \int_X \mathcal{J}\Big( \bigvee_{j=0}^n \alpha(y, S^j y)\cP \Big)(x)d\mu(x)d\nu(y)\]
	\[ = \lim_{n\to \infty} \frac{1}{n+1}\int_Y H \Big( \bigvee_{j=0}^n \alpha(y, S^j y)\cP \Big)d\nu(y)\,.
	\]
	\end{Definition}
	This expression was originally introduced by Abramov-Rokhlin \cite{AR66} in the context of skew-extensions, see the discussion in \S\S~\ref{sec:amenableeqrel}, \ref{sec:notionsentropy} below. The existence of the above limits is always guaranteed in the situations we are interested in, cf.\@ Proposition~\ref{prop:reporbitalentropy} below.
	
	\medskip

In order to relate the values of $\mathfrak{h}_\cP(\alpha,X)$ to the Rokhlin entropy, we must introduce some further assumptions.  

\medskip

{\bf Special Assumptions:}  
\begin{itemize}
\item $S$ is 
an essentially free invertible and ergodic p.m.p. transformation on a standard probability space $(Y, \nu)$, 
\item  $\alpha$ is a class injective cocycle 
$\alpha: \cO_S \to \Gamma$ defined on the orbit relation $\cO_S$, 
\item the skew extension (see \S~\ref{sec:amenableeqrel} below) $T$ of $(Y,\nu)$ by $(X,\mu)$ defined by $\alpha$ is ergodic. 
\end{itemize} 

	Define the {\em orbital Rokhlin entropy of $(X,\mu)$ w.r.t. the cocycle  $\alpha$} by: 
\begin{equation}\label{def-ent}
 \mathfrak{h}^{orb}(\alpha, X)=\inf\set{\mathfrak{h}_\mathcal{P}(\alpha, X)\,;\, \cP \text{ a generating partition of }X}\,.
\end{equation}

Assuming now in addition that the action of $\Gamma$ on $X$ is free one can use Seward's inequality  to conclude that the foregoing definition of 
entropy is equal to the Rokhlin entropy as well, as follows.

\begin{Proposition} \label{thm:ENT}
Under the freeness, injectivity and ergodicity assumptions just stated, the following holds:

\begin{equation}\label{equality-entropies} h^{\operatorname{fin}}\big( \Gamma \curvearrowright X \big)=\mathfrak{h}^{orb}(\alpha, X) = h^{\operatorname{Rok}}\big(\Gamma \curvearrowright X\big).\end{equation}

In particular, the value of $\mathfrak{h}^{orb}(\alpha, X)$ is independent of choice of the dynamical system $(Y,\nu, S)$ and the cocycle $\alpha :\mathcal{O}_S \to \Gamma$ used to define it, as long as they satisfy the assumptions stated above. Namely for any such choice, $\mathfrak{h}^{orb}(\alpha, X)$ is equal to the Rokhlin entropy of $\Gamma \curvearrowright X$. 

\end{Proposition}

\begin{proof}
Let $\cP$ be a countable partition with finite Shannon entropy. Since the 
cocycle $\alpha$ is class injective,  using subadditivity of $H$ we obtain 
\begin{eqnarray*}
\inf_{F \in \text{Fin}(\Gamma)} \frac{1}{|F|}\,H\Big( \bigvee_{\gamma \in F} \gamma\cP \Big)
\leq  \frac{1}{n+1} H\Big( \bigvee_{j=0}^n \alpha(y,S^j y) \cP \Big) \leq H\big( \cP \big)\,,
\end{eqnarray*}
for every $n \in \NN$ and for $\nu$-almost every $y \in Y$.
Note that these inequalities remain valid even if $H(\cP) = \infty$.
In the case that there exist generating partitions with finite Shannon 
entropy, we 
integrate over $Y$ and pass to the limit as $n\to \infty$. Using Proposition \ref{prop:reporbitalentropy} we conclude that $ h^{\operatorname{fin}}(\Gamma \curvearrowright X)\le \mathfrak{h}_\cP(\alpha,X)\le H(\cP)$. 
Note that this inequality holds independently of the cocycle $\alpha$, provided it satisfies the conditions stated. Now taking the infimum over all generating partitions,  since by the discussion above 
$h^{\operatorname{Rok}}(\Gamma \curvearrowright X) = h^{\operatorname{fin}}(\Gamma \curvearrowright X)$, the previous inequality  implies equality of all three notions of entropy. 
\end{proof}

Another significant consequence of this discussion is the following general result.
\begin{Theorem}\label{thm-all-groups}
Assume $\Gamma$ is any countable group, and $\Gamma$ acts ergodically and essentially freely on $(X,\mu)$ preserving the probability measure $\mu$.  
 Assume that the Rokhlin entropy of $(X,\mu)$ is finite.  
Assume there exists a dynamical system $(Y,S,\nu)$ and a cocycle $\alpha$ satisfying the conditions stated before Proposition \ref{thm:ENT}. Then for every $\epsilon > 0$ there exists a generating partition $\cP_\epsilon$, such that the Rokhlin entropy has the following (almost sure) estimate  
$$h^{\operatorname{Rok}}(\Gamma \curvearrowright X)\le  \lim_{n \to \infty} \frac{1}{n+1} \mathcal{J}\left( \cP_\epsilon^{F_n(y)}\right)(x)=H\left(\cP_\epsilon \,|\, \bigvee_{i=1}^\infty \cP_\epsilon^{F^1_i(y)}\right)\,\le h^{\operatorname{Rok}}(\Gamma \curvearrowright X)+\epsilon\,.$$
\end{Theorem}
 In particular, if $\cP$ is a generating partition with finite Shannon entropy which actually realizes the Rokhlin entropy of the $\Gamma$-action, namely $h^{\operatorname{Rok}}(\Gamma \curvearrowright X)=H(\cP)$, then the Rokhlin  entropy of $(X,\mu)$ satisfies the following remarkable (almost sure) formula :
$$h^{\operatorname{Rok}}(\Gamma \curvearrowright X)= \lim_{n \to \infty} \frac{1}{n+1} \mathcal{J}\left( \cP^{F_n(y)}\right)(x)=H\left(\cP \,|\, \bigvee_{i=1}^\infty \cP^{F^1_i(y)}\right)\,.$$

\begin{Remark} {\bf  Injectivity in the geometric set-up}. We note that for every negatively-curved group, and more generally for every hyperbolic group, there exists an ergodic dynamical system $(Y,\nu, S)$ (namely the measurable geodesic flow) with the following property. There is a cocycle $\alpha$ defined on the orbit equivalence relation defined by he powers of $S$, such that the map $n \mapsto \alpha(y, S^n y)$ is injective, for almost every $y$, namely the cocycle $\alpha$ is class-injective.

Indeed, if we choose the isometry $g_{t_0}$ such that $|t_0|> 3\,\mathrm{Diam}(\cD)$, this fact follows from the construction of the cocycle $\alpha$ for the transformation $S=g_{t_0}$ detailed in \S~\ref{sec:RManifold-1} and \S~\ref{sec:hyp}.  
\end{Remark}

\subsection{Convergence of information along random subsets}

We note that an information convergence theorem for a family of {\it random finite subsets} $F_n(y)\subset \Gamma $ for {\it general} countable groups was established by Kifer \cite{Ki86}. The sets $F_n(y)$ in question constitute the trajectories of a random walk on the group $\Gamma$, and the approach is that of ergodic theory of random transformations. This result is motivated by the important classical problem of developing random ergodic theorems for random transformations  was considered originally by Kakutani \cite{Ka51}, following a suggestion by von-Neumann and Ulam \cite{UvN47}.  Entropy theory for random transformations (and more general skew products) was developed by Abramov and Rokhlin \cite{AR66} who proved a formula for their Kolmogorov-Sinai entropy, see the discussion below for further details. It was further considered also by Morita \cite{Mo86} and by Ledrappier-Young \cite{LY88}. 

Kakutani's approach to the theory of random transformations is an instance of the more general skew extension construction in ergodic theory. Namely it corresponds to the case of a skew extension specifically of the space $(\Omega, \bf{p})$ associated with a sequence of independent random variables, by the probability space $(X,\mu)$ on which $\Gamma$ acts.
Aspects of entropy for skew product extensions for actions of $\ZZ$ were studied by Abramov and Rokhlin \cite{AR66} and also by Bogenschutz-Crauel \cite{BC90, Bo92}.  Ward and Zhang  \cite{WZ92} considered the case of skew products for amenable groups. We refer to e.g. \cite[\S~6.1]{Pe83} for an account of skew products for $\ZZ$-actions and their entropy theory. 
\medskip

The results of \S~\ref{sec:SMB} below can be viewed as a generalization of the Shannon-McMillan-Breiman theorem from the classical case of random dynamical systems given by 
 independent random variable taking values in the group, to that of general group-valued cocycles. Let us therefore formulate the results for random subsets arising from trajectories of random walks on general groups.  For background information on random walks on (Cayley graph of) groups, we refer to the monograph \cite{Wo00}. Here, we employ the set-up of \cite{Ki86}.
 
 \medskip

Let $\Gamma$ be a countable group and $m$ a probability measure on $\Gamma$ whose finite support generates $\Gamma$ as a semigroup. 
Let $(\Omega,\bf{p})$ be a standard probability space on which a sequence of independent identically distributed $\Gamma$-valued random variables $Z_i : \Omega\to \Gamma$, $ i \ge 0$ with distribution $m$ are defined. Consider the random elements of $\Gamma$ arising along a trajectory of the random walk, given for $i \ge 0$ by 
$\gamma_i(\omega)= Z_i(\omega)\cdot Z_{i-1}(\omega)\cdots Z_{0}(\omega)$.  
Define the random finite subsets $F^1_n(\omega)=\set{\gamma^{-1}_i(\omega)\,;\, 0 \le i\le n}$, and 
$F_n(\omega)=F^1_n(\omega)\cup\set{e}$. 
Consider the following partitions, the refinements of $\cP$ along (the inverse) of the trajectories of the random walk governed by the distribution $m$ :  
$$\cP^{F_n(\omega)}=\cP\vee\gamma_0(\omega)^{-1}\cP\vee \gamma_1(\omega)^{-1}\cP\vee\cdots\vee\gamma_n(\omega)^{-1} \cP\,.
$$

\begin{Theorem}\label{random-SMB}(\cite{Ki86})
{\,\,\bf Shannon-McMillan-Breiman theorem for random transformations.}
Let $\Gamma$ be countable, and $(X,\mu)$ an ergodic p.m.p.\@ action of $\Gamma$. Let $m$ be a generating measure of finite support on $\Gamma$, and $Z_i(\omega) :\Omega\to \Gamma$, $i\ge 0$ a sequence of independent $\Gamma$-valued random variables with distribution $m$. Let $\cP$ be a partition of $X$ of finite Shannon-entropy, and let $\cP^{F_n(\omega)}$ be the partitions refined using the finite random subset $F_n(\omega)\subset \Gamma$, as defined in (\ref{refine}) above. 

 The normalized information functions $\frac{1}{n+1}\cI_{\cP^{F_n(\omega)}}(x)$ converge to a limit 
for $\bf{p}$-almost every random trajectory $\omega$, and $\mu$-almost $x\in X$. The limit is given by the following limit of conditional entropies, which is almost surely a constant independent of $\omega \in \Omega$ : 
$$ 
\lim_{t\to \infty}H\left(\cP \,|\, \bigvee_{i=1}^n \cP^{F_i^1(\omega)}\right)=H\left(\cP \,|\, \bigvee_{i=1}^\infty \cP^{F^1_i(\omega)}\right).$$
\end{Theorem}
 Theorem \ref{random-SMB} is a special case of  \cite[Thm.\@ 1.5, p.\@ 63]{Ki86} (using also \cite[Lem.\@ 1.9, p.\@ 62] {Ki86} and the notation in \cite[\S~1.2, p.\@ 14]{Ki86}). 
 In \S~\ref{sec-proof-main} we will show that it is also a special case of our discussion of general cocycles. Of course, similar results can be formulated for the random walk given by $Z_0(\omega)Z_1(\omega)\cdots Z_i(\omega)$, or for the random walk defined by the variables $\check{Z}(\omega)$ whose distribution is $\check{m}$ (with $\check{m}(\gamma)=m(\gamma^{-1})$), or for measures $m$ which are not necessarily of finite support.

\subsection{From entropy convergence along almost geodesics to general ergodic theorems along genuine geodesics}
Let us note the following more general perspective on the convergence, along (almost all) almost geodesics, of information functions.  For concreteness, consider again Theorem \ref{free-gps-SMB} which asserts that for actions of the free group, the normalized information functions of the refined partitions 
$$\cP\vee\xi_0\cP\vee \xi_0\xi_1\cP\vee\cdots \vee \xi_0\xi_1\cdots\xi_{n-1} \cP\,$$
converge. A general set in the partition is an intersection $\bigcap_{i=0}^{n-1} \xi_0\cdots \xi_i A_{j_i}$, with $A_{j_i}\in \cP$. A point $x$ in the intersection has trajectory $x\in A_{j_0}, \xi_0^{-1}x\in A_{j_1}, \xi_1^{-1}\xi_0^{-1}x\in A_{j_2}, \cdots, \xi_i^{-1}\cdots \xi_0^{-1}x\in A_{j_i}$. It is therefore natural to expect that  the sequence 
of points $x, \xi_0^{-1}x, \xi_1^{-1}\xi_0^{-1}x, \cdots, \xi_i^{-1}\cdots \xi_0^{-1}x$ in $X$ has good equidistribution and ergodic properties in the space $X$, for almost all $\xi$. 
This is indeed the case for all the cases we considered above, namely other convergence  theorems - including the mean, pointwise, and maximal ergodic theorems, hold along (almost all) $C$-almost geodesics in the set-ups discussed in the present paper. This fact is an instance of the general results established in \cite{BN13b} and \cite{BN15a} about ergodic theorems along asymptotically invariant sets in general extensions of amenable ergodic p.m.p. equivalence relations using weak-mixing cocycles. 

Given the results stated above it is natural to wonder about the possibility that  convergence of normalized information functions holds along (almost all) {\it genuine geodesics} in a  hyperbolic group $H$, with respect to suitable hyperbolic metrics defined on $H$, and not just along $C$-almost geodesics. For the case of suitable word metrics it is indeed possible to establish convergence of information functions along (almost all) genuine geodesics.  For general word metrics it is possible to establish such a result for strongly mixing actions of $H$. Furthermore, convergence in the ergodic theorems just mentioned also holds along almost all genuine geodesics. These results are based on applying the arguments in the present paper and  in \cite{BN13b} and \cite{BN15a} to certain Markov codings arising from the word-metric and a suitable p.m.p. equivalence relation 
 associated with it. An exposition of these results is currently under preparation.

\section{Amenable equivalence relations, skew extensions and orbital entropy}
The approach developed in the present paper to the Shannon-McMillan-Breiman theorem for hyperbolic groups is motivated by the general approach to ergodic theorems for non-amenable groups developed in \cite{BN13b},  \cite{BN15a} and \cite{BN15b}. It is based on constructing a suitable cocycle $\alpha : \cR\to \Gamma$, from an ergodic amenable p.m.p. equivalence relation $\cR$ on a standard space $(Y,\nu)$. The cocycle produces variable finite sets in $\Gamma$ for which it is possible to derive entropy convergence  theorems. In the context of hyperbolic groups, the natural choice for such an equivalence relation arises from the $\Gamma$-orbit relation on the double boundary $\partial^2\Gamma=\left(\partial \Gamma\times \partial \Gamma\right) \setminus \Delta(\partial \Gamma)$, restricted to a set of measure $1$ w.r.t the (infinite) invariant measure.

Let us turn to describe the ingredients of this general approach.

  \subsection{Amenable equivalence relations and skew extensions} \label{sec:amenableeqrel}
We consider standard Borel equivalence relations with countable classes 
$\cR\subset Y\times Y$, where $(Y,\nu)$ is a probability space, and $[y]=\cR(y)\subset Y$ denotes the $\cR$-class of $y\in Y$. We assume that the equivalence relation preserves the probability measure $\nu$. 
Recall that $\cR$ is called hyperfinite if  it can be 
written as an increasing union of equivalence subrelations $\cR_n\subset \cR$, each 
with finite classes, i.e.
\[
\cR(y) := \bigcup_{n=1}^{\infty} \cR_n(y),\,\text{ for $\nu$-almost every } y\in Y.
\] 
 If $\cR$ is hyperfinite and ergodic, then the classes of $\cR$ are generated by a single ergodic p.m.p. transformation, almost surely. Namely, there exists an ergodic invertible measure-preserving transformation $S: Y\to Y$, such that for almost every $y\in Y$ :
$$\cO_S(y):=\set{(y, S^ky)\,;\, k \in \ZZ}=\set{(S^k y,y)\,;\, k \in \ZZ}= \cR(y)\,.$$ 
In fact, both properties are equivalent to the equivalence relation being amenable. For 
the definition of amenability and proof of the equivalence of these three conditions we refer to  \cite{CFW81}. 

\medskip
 
We recall the following definition. 
\begin{Definition}[Measurable cocycle] 
Given an equivalence relation $\cR$ with countable classes on a probability space $(Y,\nu)$, and any countable group $\Gamma$, a measurable map
\[
\alpha: \mathcal{R} \to \Gamma
\]
is called a {\em  cocycle} if for $\nu$-almost every $ y\in Y$, for all $u,v, w \in \cR(y)=[y]_\cR$ 
we have \begin{itemize}
	\item $\alpha(u,v) = \alpha(v,u)^{-1}$ and 
	\item $\alpha(u,v) \cdot \alpha(v,w) = \alpha(u,w)$.
\end{itemize}
If, in addition, there is a measurable set $Y_1 \subseteq Y$ with $\nu(Y_1) = 1$ such that 
$\alpha(y,z_1)\neq \alpha(y,z_2)$ when $y\in Y_1$ and $z_1\neq z_2$, then the 
cocycle $\alpha$ is called {\em class injective}. 
\end{Definition}

Assume that $\Gamma$ acts by p.m.p. transformations on the probability space $(X,\mu)$. Let there be given a equivalence relation $\cR$ on $(Y,\nu)$, and a measurable cocycle $\alpha : \cR\to \Gamma$. A crucial construction in our discussion is the equivalence relation, denoted $\cR^X$, which is the extension of the equivalence relation $\cR$ by the cocycle $\alpha$ and the $\Gamma$-action on $X$. 
The {\em extended equivalence relation $\cR^X$}  over
$\big( Y \times X, \nu \times \mu \big)$ is defined by the condition 
\[
\big( (y,x), (y^\prime, x^\prime) \big) \in \cR^X \Longleftrightarrow y\cR y^\prime \,\,\text{ and } x=\alpha(y,y^\prime)x^\prime
\,.\]
This well-known construction is also called the cocycle extension of $X$ by the cocycle $\alpha : \cR \to \Gamma$, or the skew extension associated with the cocycle $\alpha$.  

Assuming that the measure $\nu$ is $\cR$-invariant, it follows that the measure $ \mu \times \nu$ is $\cR^X$-invariant, since the $\Gamma$-action on $X$ preserves $\mu$. The projection map $\pi : \cR^X\to \cR$ given by $(x,y)\to y$ is {\it class-injective}, namely injective on almost every $\cR^X$-equivalence class.  
Further, it is well known that an extension of an amenable action is amenable, and thus in particular if $\cR$ is amenable, so is $\cR^X$. 

Clearly when $\cR$ is hyperfinite and  the extension is class-injective then every hyperfinite exhaustion $(\cR_n)$ of $\cR$ can be canonically lifted to a hyperfinite exhaustion $(\cR_n^X)$ of $\cR^X$, via $\cR_n^X((y,x))=\set{(y^\prime, \alpha(y^\prime,y)x)\,;\, y^\prime\in \cR_n(y)}$. 

We now define a concept of weak mixing for cocycles of the above kind.
\begin{Definition}\label{weak-mix}\cite[\S~2.2, Def.\@ 2.1]{BN15a}: 
The cocycle $\alpha : \cR \to \Gamma$ is called {\it weak-mixing} if for every ergodic p.m.p. action of $\Gamma$ on $(X,\mu)$, the extended relation $\cR^X$ is ergodic. 
\end{Definition}

 Consider now the case when $\cR$ is generated by a single given transformation $S$, so that  the relation  $\cR$ on $(Y,\nu)$ is the orbit relation of an invertible p.m.p. transformations of $S:(Y,\nu)\to (Y,\nu)$, namely $\cR=\set{(y,S^ny)\,;\, n\in \ZZ}=\cO_S$.  
The skew extension in this case is also called a skew product. It was  analyzed by Kakutani \cite{Ka51} in the context of random ergodic theorems,  and by Abramov-Rokhlin \cite{AR66} in the context of entropy theory. 
Its structure is completely determined by the associated measurable transformation $T : Y\times X\to Y\times X$, which preserves the probability measure $\nu\times \mu$, and given by : 
\begin{equation}\label{def-skew}
T(y,x)=(Sy, \alpha(Sy,y)x)\,\,;\,\, T^{-1}(y,x)=(S^{-1}y, \alpha(S^{-1}y, y)x)
\end{equation} 

Clearly the projection $p: Y\times X\to Y$ intertwines the action of $T$ on $Y\times X$ and the action of $S$ on $Y$, namely $p\circ T=S\circ p$, and $p_\ast (\nu\times \mu)=\nu$. Furthermore, the orbit relation $\cO_T$ determined by $T$ in $Y\times X$ is the extended relation $\cR^X$ defined above, namely $\cO_T=\cO_S^X$. Note that if we define $a(k, y)=\alpha (S^k y , y)$  for $y\in Y$ and $k\in \ZZ$, then for all $m,n\in \ZZ$ the cocycle identity implies  that $a$ satisfies the cocycle equation: 
\begin{equation}\label{cocycle-S}
a(n+m,y)= \alpha(S^{n+m}y,  y)=\alpha(S^{n+m}y, S^m y)\alpha(S^m y, y)=a(n,S^m y) a(m,y)\,.
\end{equation}

\subsection{Notions of entropy for skew extensions} \label{sec:notionsentropy} 

Let us now consider several of the notions of entropy that have played a role in the theory of skew products and more generally skew extensions.

\subsubsection{Abramov-Rokhlin relative entropy for skew products}
In their study of entropy for skew-product transformation,  Abramov-Rokhlin \cite{AR66}, defined the relative entropy of the system.  They consider the measure preserving extension of  $(Y,\nu)$ by $(X,m)$ via the cocycle $a : \ZZ\times Y\to \text{Aut}(X,\mu)$, 
taking values in the group of measure-preserving transformations on $X$, and show that the following limit exists, for any countable partition $\cP$ of $X$ of finite entropy:  
\begin{equation}\label{AR}  \lim_{n\to \infty} \frac{1}{n+1}\int_Y H \Big( \bigvee_{j=0}^n \alpha(y, S^j y)\cP \Big)d\nu(y)\,.\end{equation}
  The {\it relative entropy} $h^{rel}(T,S)$ is defined in \cite{AR66} by taking the supremum over all finite partitions $\cP$ of $ X$ (see \cite[\S 6.1]{Pe83} for a full discussion) and denote by $h^{rel}(T,S)$. For a detailed account of the relative fibrewise entropy of a skew product transformation and its properties we refer to  \cite[Ch. 6]{Pe83} (see also \cite{Ki86} for the case of independent random variables).

 Denote by $h^{\text{KS}}$ the Kolmogorov-Sinai entropy of an invertible p.m.p. transformation on a probability space. For the associated p.m.p. transformation $T: Y\times X \to Y\times X$ it is given by:  
\begin{Theorem}\label{AR formula}(Abramov-Rokhlin formula \cite{AR66})

\begin{equation}\label{AR-sum}
h^{\text{KS}}(T,Y\times X)=h^{\text{KS}}(S,Y)+ h^{rel}(T,S).
\end{equation}
\end{Theorem}
For a given a  partition $\cP$ of $X$ the expression (\ref{AR}) introduced in \cite{AR66} is denoted in our notation by $\mathfrak{h}_\cP(\alpha, X)$.
  We note  that the orbital Rokhlin entropy $\mathfrak{h}^{orb}(\alpha, X)$ introduced above is completely different than the relative entropy 
$h^{rel}(T,S)$ appearing in the Abramov-Rokhlin formula. Indeed, $h^{rel}(T,S)$
is defined as the {\it supremum of (\ref{AR}) over all finite partitions} $\cP$ of $X$.
This is a very different expression than $\mathfrak{h}^{orb}(\alpha, X)$ which is defined as the {\it infimum of (\ref{AR}) over all $\Gamma$-generating partitions} of $X$. (In fact, by Seward's analog of Krieger's theorem \cite{Se15a}, \cite{Se15b} the infinimum may be taken over finite $\Gamma$-generating partitions).  

The definition of  $\mathfrak{h}^{orb}(\alpha, X)$ we introduce is of course motivated by the remarkable insight that the key to developing entropy theory beyond amenable groups 
is the passage from the {\it supremum} over all finite partitions of the space (which is the  customary definition of the Kolmogorov-Sinai entropy) to the {\it infimum} over countable  generating partitions. 
This ground-breaking principle originated in L. Bowen's work developing the definition and properties of the $f$-entropy for free groups and sofic entropy for sofic groups \cite{Bo10c}\cite{Bo12}.  This principle was greatly enhanced by B. Seward's definition of Rokhlin entropy for general groups \cite{Se15a} and the theory of Rokhlin entropy based on it \cite{Se15a}, \cite{Se15b},\cite{Se16}.

\medskip

 We also point out that various ergodic and entropy theoretic aspects of measure preserving skew products have been studied, see for instance investigations on relative topological pressure \cite{Wa86}, or Smale endomorphisms
 \cite{MU20}.

\subsubsection{Orbital entropy for amenable groups}
An orbital approach to entropy theory for actions of amenable groups was initiated  by 
Rudolph and Weiss \cite{RW00} and developed further by Danilenko \cite{Da01} and Danilenko-Park \cite{DP02}. We recall the set-up of orbital entropy theory developed there. 
Assume that $\Gamma$ is a countable group which acts freely by probability measure preserving transformations on a probability space $(X,\mu)$. Further suppose that $\cR$ is a hyperfinite measure preserving Borel equivalence relation on the probability space $(Y,\nu)$ and that there is a class injective cocycle $\alpha: \cR \to \Gamma$. Let $\cP$ denote a finite measurable partition of $X$.   

\begin{Proposition} \label{prop:entropy}\cite{Da01}.
	For every bounded, hyperfinite exhaustion $(\cR_n)$, we have 
	\begin{equation}\label{ent-hf}
	 \lim_{n \to \infty}  \int_Y \frac{H\Big( \bigvee_{z \in \cR_n(y)} \alpha(y,z) \cP \Big)}{|\cR_n(y)|}\,d\nu(y)=\inf_{\cT} \int_{Y} \frac{H\Big( \bigvee_{z \in \cT(y)} \alpha(y,\, z)\cP\Big)}{|\cT(y)|}\,d\nu(y),
	\end{equation}
	
	where the infimum runs over all bounded subequivalence relations $\cT$ of $\cR$.
\end{Proposition}
Orbital entropy itself is defined in \cite{Da01} as the {\it supremum} of the foregoing expressions over all finite partitions of the space. Note that every ergodic p.m.p. action of every amenable group $\Gamma$ is amenable, the orbit relation it generates is hyperfinite, and by definition it has a class-injective cocycle into $\Gamma$.  Orbital entropy theory developed by \cite{RW00}, \cite{Da01} and \cite{DP02} focuses on orbit relations of actions of amenable groups via the study of the foregoing expression for the hyperfinite equivalence relations these actions generate. This is demonstrated in the papers just cited to be an effective tool for the study of the Kolmogorov-Sinai entropy for actions of amenable groups. 

We note however that \cite{Da01} and \cite{DP02} do not consider the problem of pointwise convergence of information functions for partitions refined using classes of increasing finite relations. This convergence result was established for general  cocycles on arbitrary hyperfinite  equivalence relations in \cite{NP17}. 

In the case where $\Gamma$ is non-amenable, the definition of orbital entropy in \cite{Da01} just described, namely using the supremum over finite partitions, will typically yield for a properly ergodic action of $\Gamma$ the value $\infty$.
The same applies to the definition using the supremum over finite partitions in a skew product of an action of the free group over a Bernoulli shift given by R. Grigorchuk in \cite{Gr99}. However, the definition using the infimum over finite generating partitions gives, in all cases, the Rokhlin entropy of the action. 

\subsubsection{Orbital entropy for actions of non-amenable groups}
For weak-mixing class-injective cocycles, in \cite{NP17} orbital Rokhlin entropy was defined as the {\it infimum} of (\ref{ent-hf}) 
over generating partitions. It was shown, by an argument similar to Proposition \ref{thm:ENT} above, that for free ergodic actions, it equals the Rokhlin entropy of the $\Gamma$-action on $X$, as defined by Seward. A Shannon-McMillan-Breiman theorem for the pointwise convergence of the information functions of partitions of refined along the classes of increasing finite relations was established in \cite{NP17}.

Given an amenable p.m.p.\@ equivalence relation $\cR$ on a space $Y$, there are of course many different ways to  generate its countable equivalence classes. The two most natural ones are hyperfinite exhaustions, and generating the classes by the orbits of a single p.mp. transformation of the underlying probability space. As noted above it was  shown by \cite{CFW81}  that a p.m.p. equivalence relation is amenable if and only if it admits a hyperfinite exhaustion, if and only and its classes can be generated by a single invertible p.m.p. transformation. Such a transformation is of course far from unique, and our discussion focuses on the orbit relation which is generated by the action of fixed  judiciously chosen invertible p.m.p. transformations $S$ on  suitable probability spaces $(Y,\nu)$.  This transformation will be chosen to produce natural variable geometric sets in the group, namely segments of almost geodesics in the group, as described in the discussion of the geometric cocycles in the previous section.  

We note the remarkable fact that the two alternative definitions (using hyperfinite exhaustions or orbits of a single m.p. map) of orbital Rokhlin entropy $\mathfrak{h}(\alpha,X)$ 
give one and the same invariant, independently of whether the classes of the relation $\cR$ are exhausted by increasing partial orbits of a p.m.p. transformation, or by  increasing finite relations. In all cases the value is equal to the Rokhlin entropy of the action of $\Gamma$ on $(X,\mu)$.
 As noted above, this is a consequence of Seward's inequality, see Proposition \ref{thm:ENT} for p.m.p.\@ transformations, and \cite[Prop.\@ 1.7]{NP17}.

\begin{Remark}\label{notation convention}{\bf Notational convention.}
For brevity, we denote by $h_{\mathcal{P}}(\alpha, X)$ the {\em orbital entropy} of the partition $\cP$ of $X$ with respect to the given p.m.p. action of $\Gamma$ on $X$ and the cocycle $\alpha : \cR \to \Gamma$. The partition $\cP$ is refined using the finite sets $F_n(y)\in \text{Fin }(\Gamma)$, given by $F_n(y)=\set{\alpha(y, z)\,;\, z \in \cF_n(y)}$ where $\cF_n(y)$ are finite sets with $\set{y}\times \cF_n(y) \subset \cR(y)$. In the short-hand notation $h_{\cP}(\alpha,X)$ for the entropy, $\alpha$ stands for the data $(Y, \nu, \cR, \Gamma, \alpha : \cR\to \Gamma, \cF_n(y))$ and $X$ stands for the data $(X,\mu, \Gamma \curvearrowright X)$. 
\end{Remark}

\subsection{Comparison to other notations} 
In \cite[p.\@ 233]{Pe83}, the action of a m.p. map $T: X \to X$ on a partition $\cP=\set{A_1, \dots , A_k}$ is $T^{-1}\cP=\set{T^{-1}A_1, \dots, T^{-1}A_k}$. Skew products are defined on p. 254 via a map $x\to S_x$, from $X$ to m.p. maps $S_x: Y\to Y$. The skew products $\tilde{T}:  X \times Y\to X \times Y$ is given by $\tilde{T} (x,y)=(Tx, S_x y)$.  
So for $n > 0$ $\tilde{T}^n(x,y)=(T^n x, S_{T^nx}S_{T^{n-1}x} \dots S_x y)$. 
For a partition $\cP$ of $Y$, the refined partitions are defined by 
$$\beta_1^n(x)=S_x^{-1}\cP \vee S_x^{-1} S_{Tx}^{-1} \cP\vee\cdots\vee S_x^{-1}S_{Tx}^{-1}\cdots S_{T^{n-1}x}^{-1}\cP$$

In the set-up of the present paper, we have $S: Y\to Y$ and a cocycle $c : \ZZ\times Y \to \Gamma$ where $\Gamma$ is a group of m.p. transformations of $X$. The skew product transformation is $T :Y \times X \to Y\times X$ given by $T(y, x)=(Sy, c(1,y)x)$, and so 
$T^n(y,x)=(S^ny, c(1, S^{n-1}y)\dots c(1, Sy)c(1,y)x)$. The relation is the orbit relation of $S$ on $Y$ namely $\cR=\cO_S$, and $\alpha(S^n y,y)=c(n,y)$ as noted in (\ref{cocycle-S}) above. 
Therefore, $y\mapsto c(1,y) : X \to X $ is the map assigning to $y\in Y$ an m.p.t of $X$, and so if $\cP$ is a partition of $X$, then the partitions just defined are in our notation 
$$c(1,y)^{-1}\cP\vee c(1,y)^{-1}c(1,Sy)^{-1}\cP \vee \cdots \vee c(1,y)^{-1}c(1,Sy)^{-1}\cdots c(1,S^{n-1}y)^{-1}\cP
$$
$$=\alpha(y, Sy)\cP\vee \alpha(y, Sy)\alpha(Sy, S^2y)\cP\vee \cdots \vee \alpha(y, Sy)\alpha(Sy, S^2y)\cdots \alpha(S^{n-1}y, S^{n}y)\cP
$$
$$=\alpha(y, Sy)\cP\vee \alpha(y,S^2y)\cP\vee \cdots \vee \alpha(y, S^{n}y)\cP
$$
$$=\bigvee_{z \in \cF_n(y)} \alpha(y, z)\cP\,\,\,;\,\,\, \cF_n(y)=\set{Sy, S^2y,\dots ,S^n y}
$$
provided $\cP^\gamma=\gamma\cP=\set{\gamma A_1, \dots \gamma A_k}$.  

We also note that in \cite[p.\@ 2]{Bo92}  the expression occuring in Thm.\@ 2.2 stated there for the refined partitions used to compute the fibrewise entropy is the same : 
$$\cP\vee c(1,y)^{-1}\cP\vee c(2,y)^{-1}\cP\vee\cdots\vee c(n-1,y)^{-1}\cP\,.
$$

\section{Fibrewise entropy for skew extensions}

\subsection{Preliminaries : conditional entropy and information functions}
We begin by recalling some of the basic definitions and properties of conditional entropy and information functions, which will be used repeatedly in the proofs below. We refer to  \cite[Ch. 6]{Si94} for a full discussion.

\medskip

Suppose that $(X,\mu)$ is a probability space. Let $\cP$ be two countable measurable partitions $\cQ$ with $H(\cP)+H(\cQ) < \infty$. The {\it Shannon entropy} of the partition $\cP$ is defined by 
$$H(\cP)=-\sum_{i\in \NN} \mu(A_i) \log \mu(A_i)\ge 0\,.$$ 
The unique cell of the partition $\cP$ containing $x$ is denoted $\cP(x)$. 
If $\cQ$ is another countable partition, $\cQ$ refines $\cP$ if every cell of $\cP$ is a union of cells of $\cQ$, and this is denoted $\cQ \ge \cP$. 

The {\it conditional probability} of $\cP$ given $\cQ$ is the measurable function on $X$ defined by 
\begin{equation}\label{cond-prob}
\mu\big( P\,|\, \mathcal{Q} \big)(x) = \sum_{Q \in \mathcal{Q}} \frac{\mu(P \cap Q)}{\mu(Q)}\cdot \one_Q(x).
\end{equation}
The {\it conditional information function} for $\cP$ conditioned on $\cQ$ is defined by 
\begin{equation}\label{cond-info}
\mathcal{J}(\cP|\cQ)(x) := -\sum_{P \in \cP} \log \mu\big( P\,|\, \mathcal{Q} \big)(x) \cdot \one_P(x)\,.
\end{equation}
Note that if $P \in \cP$ and $Q \in \cQ$ and $x \in P \cap Q$, then 
\begin{eqnarray} \label{eqn:condinffunction}
\mathcal{J}\big( \cP\,|\, \cQ \big)(x) = - \log \frac{\mu(P \cap Q)}{\mu(Q)}.
\end{eqnarray}

\medskip

The {\it conditional Shannon entropy} of $\cP$ given $\cQ$ is defined by 
\begin{equation}\label{cond-entr}
H\big( \cP\,|\, \cQ \big) := - \sum_{P \in \cP} \sum_{Q \in \cQ} \log \mu(P\,|\, Q) \cdot \mu(P \cap Q)\,. 
\end{equation}
Note that the conditional entropy of $\cP$ given $\cQ$ is the integral over $X$ of the conditional information functions of $\cP$ given $\cQ$, by  equality~\eqref{eqn:condinffunction} :  
\begin{equation}\label{cond-ent-int}
	H\big( \cP\,|\, \cQ \big) = \int_X \mathcal{J}\big( \cP\,|\, \cQ \big)(x)\,d\mu(x).
	\end{equation}

\begin{Proposition}\label{prop-cond-ent}{\it (Properties of conditional entropy).}
Let $\cP_1, \cP_2$ and $\cQ_1, \cQ_2$ be two countable partitions of the standard probability space $(X,\mu)$ with finite Shannon entropy.  
\begin{enumerate}
\item If $\gamma : X \to X$ is a p.m.p. map then 
$H(\gamma^{-1} (\cP\vee \cQ))= H(\gamma^{-1} \cP\vee \gamma^{-1}\cQ)=H(\cP\vee \cQ)$
\item $H(\cP\vee\cQ)=H(\cP)+H(\cQ\,|\,\cP)\le H(\cP)+H(\cQ)$. 
\item $H(\cP\vee \cQ_1 | \cQ_2)=H(\cP| \cQ_2)+H(\cQ_1| \cP\vee \cQ_2)\le H(\cP | \cQ_2)+H(Q_1| \cQ_2)$.
\item $H(\cP|\cQ_1)\le H(\cP|\cQ_2)$ if $\cQ_1\ge \cQ_2$ namely $\cQ_1$ refines $\cQ_2$.  
\end{enumerate}
\end{Proposition}

Let $\cP_i$ and $\cQ_i$, $i\in \NN$ be two sequences of countable partitions of the standard probability space $(X,\mu)$ with finite Shannon entropy. Denote by $\cP^\vee=\bigvee_{i\in \NN}\cP_i$ the smallest partition refining each of the partitions $\cP_i$, $i\in \NN$, and by $\cP^\wedge=\bigwedge_{i\in \NN} \cP_i$ the largest partition refining each of the partitions $\cP_i$, $i\in \NN$. 

\begin{Proposition}\label{prop-cond-ent}{\it (Convergence of conditional entropy).}
\begin{enumerate}
\item If $\cP_1\le \cP_2 \le \cdots \le \cP_i\le \cdots $ is an increasing sequence of partitions, then 
$$\lim_{i\to \infty} H(\cP_i | \cQ)=H(\cP^\vee | \cQ)\,.$$
\item If $\cP_1\ge  \cP_2 \ge \cdots \ge \cP_i\ge \cdots $ is an decreasing sequence of partitions, then 
$$\lim_{i\to \infty} H(\cP_i | \cQ)=H(\cP^\wedge | \cQ)\,.$$
\item If $\cQ_1\le \cQ_2 \le \cdots \le \cQ_i\le \cdots $ is an increasing sequence of partitions, then 
$$\lim_{i\to \infty} H(\cP | \cQ_i)=H(\cP | \cQ^\vee)\,.$$

\end{enumerate}
\end{Proposition}

\subsection{Pointwise convergence of fibrewise entropy}

We now turn to discuss the fibrewise entropy of a skew extension.  Convergence of the fiberwise entropy (\ref{limit-S}) for a skew product is due to Abramov-Rokhlin \cite{AR66} (see also the account in \cite{Pe83} for skew products), and to Ward-Zhang \cite{WZ92} for actions of amenable groups more generally. Morita \cite[Thm.\@ 1]{Mo86} establishes almost sure convergence of the fibrewise entropy (\ref{Fibrewise}). We will follow Petersen \cite[Ch. 6, p. 254, Prop.\@ 1.1]{Pe83}, and  Bogensch\"utz \cite[Thm.\@ 2.2 and Thm.\@ 4.2~(ii)]{Bo92}  (who both consider only finite partitions), and give some details of the proof including the computation of the limit (\ref{formula-cond}), for the sake of completeness.

\begin{Proposition} \label{prop:reporbitalentropy}
Let $\Gamma$ be a countable group acting by probability measure preserving transformations on $(X,\mu)$. Let $S$ be an invertible p.m.p. transformation $S$ on $(Y,\nu)$, and let $\alpha: \cO_S \to \Gamma$ be a cocycle.  Let $\cP$ be a countable measurable partition of $(X,\mu)$ with $H(\cP)< \infty$. Then the sequence of  functions 
	\begin{equation}\label{Fibrewise}
 \frac{1}{n+1} H\Big( \bigvee_{j=0}^n \alpha(y, S^j y) \cP \Big)
	\end{equation}
	converges for $\nu$-almost every $y \in Y$ to a limit.  If $S$ is assumed to be ergodic, then the limit is a constant almost surely,  
	given by the following limit 
	\begin{equation}\label{limit-S}
	 \lim_{n \to \infty} \frac{1}{n+1} \int_Y H\Big( \bigvee_{j=0}^n \alpha(y,\, S^j y) \cP \Big)\, d\nu(y):=
	 \mathfrak{h}_\cP(\alpha,X)\,.
	\end{equation}
Its value is given by the average of the conditional entropies, namely  
	\begin{equation}\label{formula-cond} 
	 \mathfrak{h}_\cP(\alpha,X)=
 \int_Y H\Big( \cP\,|\, \bigvee_{j=1}^{\infty} \alpha(y,S^j y) \cP \Big)\,d\nu(y).
\end{equation}
			 
\end{Proposition}
\begin{proof}
We define $g_n(y)=H\Big( \bigvee_{j=0}^{n-1} \alpha(y,S^j y)\cP\Big)$. Combining the subadditivity property of Shannon entropy with the cocycle identity for $\alpha$, a short computation shows $g_{n+m}(y) \leq  g_n(y)+g_m(S^n y)$. Since $H(\cP) < \infty$, we have $g_n \in L^1(Y,\nu)$
for all $n \in \NN$. 
Therefore the existence of the limits, for $\nu$-almost every $y \in Y$
	\[
 \lim_{n \to \infty} \frac{1}{n+1} H\Big( \bigvee_{j=0}^n \alpha(y,S^j y) \cP \Big)
	\] 
follows from Kingman's ergodic theorem.  If $S$ is also ergodic, then the limit function is constant almost-surely, since it is easily seen, by the cocycle property, to be invariant under $S$. 
	Since $0\le \frac1n g_n(y)\le H(\cP)$, Lebesgue's dominated convergence theorem is applicable, and so this constant is $\mathfrak{h}_\mathcal{P}(\alpha, X)$.  Furthermore, a straightforward induction as in  \cite[Ch. 6]{Pe83} shows that for almost every $y \in Y$ and all $n \in \NN_{\geq 0}$, we have 
	\begin{eqnarray} \label{eqn:Cesaro}
	 H\Big( \bigvee_{j=0}^n \alpha(y,S^j y)\cP \Big)=  \sum_{k=0}^n H\Big( \cP\,|\, \bigvee_{j=1}^k \alpha(S^{n-k}y, S^{n-k} \, S^j y) \cP \Big),
	\end{eqnarray}
	where we use the convention $\bigvee_{j=1}^0 \alpha(y, S^j y) \cP = \{\emptyset, X\}$ for all $y$.

	The proof of this claim proceeds by induction, noting that the above relation is correct for $n = 0$ since both sides of the equation give the value $H(\cP)$. Now let $n \geq 1$ and assume that the equality~\eqref{eqn:Cesaro} holds for $n-1$. 
	Using Proposition \ref{prop-cond-ent} (for $\cP$ and $ \bigvee_{j=1}^n \alpha(y,S^j y)\cP$)  we have for $\nu$-almost every $y$ 
	\[
	H\Big( \bigvee_{j=0}^n \alpha(y, S^j y)\cP \Big) = H\Big( \bigvee_{j=1}^{n} \alpha(y, S^j y)\cP \Big) +  H\Big(  \cP \,\Big|\, \bigvee_{j=1}^{n} \alpha(y, S^j y)\cP  \Big)\,.
	\]
	Therefore, since $\alpha(y,Sy)\in \Gamma$ and $\Gamma $ acts on $(X,\mu)$ by invertible measure-preserving maps, by Proposition \ref{prop-cond-ent}
	\begin{eqnarray*}
		 H\Big( \bigvee_{j=0}^n \alpha(y, S^j y)\cP \Big)  &=& H\Big(\alpha(Sy,y) \bigvee_{j=1}^{n} \alpha(y, S^j y)\cP \Big)  + \,  H\Big(  \cP \,\Big|\, \bigvee_{j=1}^{n} \alpha(y, S^j y)\cP  \Big) \\ 
		&=& H\Big( \bigvee_{j=1}^{n} \alpha(Sy, y) \cdot \alpha(y, S^j y)\cP \Big)+ \, H\Big(\cP\,|\, \bigvee_{j=1}^n \alpha(y,S^j y) \cP \Big) \\
	\mbox{ (cocycle identity) }	 &=&  H\Big( \bigvee_{j=0}^{n-1}  \alpha(Sy, S^{j} Sy)\cP \Big) \, + \, H\Big(\cP\,|\, \bigvee_{j=1}^n \alpha(y,S^j y) \cP \Big) \\
\mbox{ (by induction) }	&=& \sum_{k=0}^{n-1}  H\Big( \cP\,|\, \bigvee_{j=1}^k \alpha( S^{n-1-k}\, Sy,\, S^{n-1-k}\,S^j\,S y) \cP  \Big)\,d\nu(y)  + \,H\Big(\cP\,|\, \bigvee_{j=1}^n \alpha(y,S^j y) \cP \Big)\\
&=&\sum_{k=0}^n H\Big( \cP\,|\, \bigvee_{j=1}^k \alpha(S^{n-k}y, S^{n-k} \, S^j y) \cP \Big)\,.
		\end{eqnarray*}

	The sequence
	 \[
	 c_k(y) = H\Big( \cP\,|\, \bigvee_{j=1}^k \alpha(y,S^j y)\cP\Big)\le H(\cP)
	 \]
	 is decreasing in $k$ for a given $y$ by Proposition \ref{prop-cond-ent}, and furthermore the following limit exists for $\nu$-almost every $y\in Y$ and satisfies the inequality 
	\[
	\lim_{k \to \infty} c_k(y) = H\Big( \cP\,|\, \bigvee_{j=1}^{\infty} \alpha(y,S^j y)\cP\Big)\le H(\cP)\,.
	\]
	
	We now integrate the equality~\eqref{eqn:Cesaro} over $Y$. Using the fact that $S$ preserves $\nu$, the uniform bound $0\le c_k(y)\le H(\cP)$,  and Lebesgue's dominated convergence theorem, the foregoing limit relation yields 
	\begin{eqnarray*}
	&& \lim_{n \to \infty} \frac{1}{n+1}\int_Y H\Big( \bigvee_{j=0}^n \alpha(y,S^j y)\cP \Big)\,d\nu(y) = \lim_{n \to \infty} \int_Y \frac{1}{n+1} \sum_{k=0}^n c_k(y)\, d\nu(y) \\
	 && \quad = \int_Y \lim_{n \to \infty} \frac{1}{n+1} \sum_{k=0}^n c_k(y) \, d\nu(y) =  \int_Y H\Big( \cP\,|\, \bigvee_{j=1}^{\infty} \alpha(y,S^j y)\cP\Big)\, d\nu(y),
	\end{eqnarray*}
	since the Ces{\`a}ro averages of a convergent sequence converge to the limit of that sequence. 
\end{proof}

\section{The Shannon-McMillan-Breiman theorem for skew extensions} \label{sec:SMB}

As before, we assume that a countable group $\Gamma$ acts on a probability space $(X,\mu)$ by measure preserving transformations. Moreover, we suppose that $S$ is a measure preserving, essentially free transformation on a probability space $(Y,\nu)$ and that there is a measurable 
 cocycle $\alpha: \cR \to \Gamma$. Recall that this gives rise to an associated probability measure preserving transformation $T$ on the product space $Y \times X$, defined as
\[
T: (Y \times X,\, \nu \times \mu) \to (Y \times X,\, \nu \times \mu), \,\, T(y,x) := \big( Sy,\, \alpha(Sy,y)x \big)\,
\] 
The condition that for all  ergodic actions of $\Gamma$ on $(X,\mu)$, the arising $T$ is ergodic w.r.t. $\nu\times \mu$ is equivalent to the condition that the cocycle $\alpha$ is weak-mixing. 

\medskip

In this section, we prove a 
Shannon-McMillan-Breiman theorem subject to the condition that $T$ is ergodic, and establish pointwise convergence to the limit given by the relative fibrewise entropy. 

\begin{Theorem}[Shannon-McMillan-Breiman theorem for relative entropy of skew extensions] \label{thm:SMBmain}
	Let $\Gamma$ be any countable group which acts ergodically by probability measure preserving transformations on $(X,\mu)$.
	Assume that $S$ is an essentially free invertible measure-preserving transformation on a probability space $(Y,\nu)$. Let $\cR=\cO_S$ be the orbit relation generated by $S$, and assume that there is a measurable cocycle $\alpha: \cR \to \Gamma$. 
	If the skew product transformation $T$ on $(Y \times_\alpha X, \nu \times \mu)$ defined by $S$ and $\alpha$ is ergodic, then for all countable partitions $\mathcal{P}$ of $X$ with $H(\mathcal{P}) < \infty$, 
	\begin{eqnarray*}
		\mathfrak{h}_\mathcal{P}(\alpha, X) = \lim_{n \to \infty} \frac{\mathcal{J}\big( \bigvee_{j=0}^n \alpha\big(y, S^j y \big)\cP  \big)(x)}{n+1}
	\end{eqnarray*}
	for $(\nu \times \mu)$-a.e.\@ $(y,x) \in Y \times X$. Moreover, the convergence also holds in $L^1(Y \times X, \nu \times \mu)$.
\end{Theorem}

In the case of independent random transformations this result is due to  Kifer \cite{Ki86} (see \S~\ref{sec:random} below for further discusion).  Bogensch\"utz \cite{Bo92} proves a more general result, where the fibrewise measure on the skew product is not necessarily constant as in our case, but considers only finite partitions. 
Zhu \cite{Zhu08} also proves a pointwise convergence theorem for random dynamical systems in the context of continuous measure preserving maps on a compact metric space, but the proof does not apply to actions on more general standard probability measure spaces. Finally, we note that a result stated by Morita \cite[Thm.\@ 3]{Mo86} implies Theorem \ref{thm:SMBmain}, but the sketch of the proof given there is very terse. We elaborate the proof not only for the purpose of self-containment, but also because the arguments are necessary in order to give a full proof of the maximal inequalities related to the Shannon-McMillan-Breiman theorem.  These maximal inequalities are necessary for the proof of the integrated form of the Shannon-McMillan-Breiman theorem stated in Corollary \ref{cor:SMBint}, which underlie the proof of Corollary \ref{free-gps-SMB-int}.

\subsection{Proof of the Shannon-McMillan-Breiman theorem}
This subsection is devoted to the proof of Theorem~\ref{thm:SMBmain} and its consequences.
We start by fixing a countable partition $\cP$ of $X$ with 
$H(\cP) < \infty$. With no loss of generality we assume $\mu(P) > 0$ for all $P \in \cP$. 
 In the result that follows, we will consider general rooted nested subset functions, namely sequences  
$y\in \cF_n(y)\subset \cF_{n+1}(y)\subset \cR(y)$ with $\cup_{n\in \NN} \cF_n(y)=\cR(y)$ for $\nu$-almost all $y\in Y$. Our main interest will be in the following two examples :
\begin{itemize} 
\item  $\cF_n(y)=\cR_n(y)$ where $\cR=\cup_{n \in \NN} \cR_n$ is a hyperfinite exhaustion, \item
 $S: Y \to Y$ is a p.m.p. transformation and $
\cF_n(y) := \big\{ S^j y:\, 0 \leq j \leq n \big\}, \quad y \in Y.$
\end{itemize}
 We further set 
$\cF_n^1(y) := \cF_n(y)\setminus \set{y}$ for $n \in \NN_{\geq 0}$ and $y \in Y$, as well as 
$\cF_{\infty}^1(y) := \cup_{n\in \NN}\cF_n^1(y)$.

\begin{Lemma} \label{lemma:integrability}
	Define $f_n(y,x) := \mathcal{J}\big( \cP \,|\, \bigvee_{z \in \cF_n^1(y)} \alpha(y,z)\cP \big)(x)$ and 
	$f^{*}(y,x) := \sup_{n \in \NN_{\geq 0}} f_n(y,x)$.
	Then the following holds.
	\begin{enumerate}[(a)]
		\item for each $\lambda \geq 0$ and each $P \in \cP$, we have 
		\[
		(\nu \times \mu)\big( \big\{ (y,x) \in Y \times P:\, f^{*}(y,x) > \lambda  \big\} \big) \leq e^{-\lambda}.
		\]
		\item  Fix $p$ satisfying $1 \leq p < \infty$, and assume that the partition $\cP$ satisfies 
		\[
		\sum_{P \in \cP} \mu(P) \big( - \log \mu(P) \big)^{p}  < \infty. 
		\]
		Then $f^{*} \in L^{p}(Y \times X,\, \nu \times \mu)$.
		In particular, $f^{*} \in L^1(Y \times X,\, \nu \times \mu)$. Moreover, if $\cP$ is finite, then $f^{*} \in L^q (Y \times X,\, \nu \times \mu)$ for all $1 \leq q < \infty$. 
		\item The sequence $(f_n)$ converges $(\nu \times \mu)$-almost everywhere and in $L^1$-norm to the limit 
		\[
		f^{\infty}(y,x) = \mathcal{J}\Big(\cP\,|\, \bigvee_{z \in \mathcal{F}^1_{\infty}(y)} \alpha(y,z)\cP\Big)(x).
		\]
	\end{enumerate}
\end{Lemma}

\begin{Remark}
	We point out that if the integrability condition in statement~(b) of the previous lemma is satisfied for some $p \geq 1$, then we necessarily have $H(\cP) < \infty$. This follows from 
	\begin{align*}
	H(\cP) &= \sum_{P \in \cP: \mu(P) \geq e^{-1}} -\mu(P) \log \mu(P) \, + \, \sum_{P \in \cP: \mu(P) < e^{-1}} \mu(P) (- \log \mu(P)) \\
	&\leq   \sum_{P \in \cP: \mu(P) \geq e^{-1}} -\mu(P) \log \mu(P) + \sum_{P \in \cP} \mu(P) (- \log \mu(P))^p,
	\end{align*}
	while noting that the first sum consists of finitely many terms only. 
\end{Remark}

\begin{proof}
	We first prove the statement~(a). 
	For every atom $P \in \mathcal{P}$ and $n \geq 0$, we set 
	\[
	f_n^P(y,x) :=  \mathcal{J}\Big( \cP\,|\, \bigvee_{z \in \cF_n^1(y)} \alpha(y,z)\cP \Big)_{|P}(x) = - \log\, \mu\big( P\,|\, \bigvee_{z \in \cF_n^1(y)} \alpha(y,z)\cP \big)(x)
	\]
	for $(y,x) \in Y \times X$,  where we use the convention that $f_0^P(y,x) = - \log \mu(P) $.
	Define also  
	\[
	B_n^P(y) := \big\{ x \in X:\, f^P_{\ell}(y,x) \leq \lambda \mbox{ for all } 0 \leq \ell < n \mbox{ and } f^P_n(x,y) > \lambda \big\}
	\]
	for $y \in Y$. Note that we have $f_n^P(y,x) = f_n(y,x)$ on $Y \times P$. Using the fact that for each $k \leq n$ and all $y \in Y$, we have 
	$B_k^P(y) \in \sigma\big( \bigvee_{z \in \cF_n^1(y)} \alpha( y, z) \cP \big)$ (namely $B_k^P(y)$ belongs to  the $\sigma$-algebra generated by the sets in the partition  $\bigvee_{z \in \cF_n^1(y)} \alpha( y, z) \cP$ with the convention that $B_0^P(y) \in \{\emptyset,\, X\}$), we compute
	\begin{eqnarray*}
		\mu\big( B_k^P(y) \cap P \big) &=& \int_{B_k^P(y)} \one_P\, d\mu = \int_{B_k^P(y)} \mu\big( P\,|\, \bigvee_{z \in \cF_n^1(y)} \alpha(y, z)\cP \big) d\mu \\
		&=& \int_{B_k^P(y)} e^{-f_k^P(y,\cdot)}\,d\mu \leq e^{-\lambda} \cdot \mu\big(B_k^P(y) \big). 
		\end{eqnarray*}
	for $y \in Y$. Note that by definition of the events $B_k^P(y)$, we have 
	\[
	\big\{ x \in P:\, f^{*}(y,x) > \lambda \big\} \, = \, \bigcup_{k=0}^{\infty} \big( B_k^P(y) \cap P \big)
	\] 
	for $y \in Y$. Combining the pairwise disjointness of the sets $B_k^P(y)$ with the previous estimate, we obtain
	\[
	\mu\big( \big\{ x \in P:\, f^{*}(y,x) > \lambda \big\} \big) = \sum_{k=0}^{\infty} \mu\big( B_k^P(y) \cap P \big) \leq e^{-\lambda} \cdot \sum_{k=0}^{\infty} \mu\big( B_k^P(y) \big) \leq e^{-\lambda}
	\]
	for $y \in Y$. 
	Integrating over $Y$ concludes the proof. \\
	We now use part~(a) in order to prove the assertion~(b). Note that it follows from part~(a) that 
	\[
	\nu \times \mu\big( \big\{ (y,x) \in Y \times P:\, f^{*}(y,x) > \lambda \big\} \big) \leq \min\{ \mu(P);\, e^{-\lambda} \}.
	\] 
	Now fix $p$ satisfying  $1\le p < \infty$. We compute
	\begin{eqnarray} \label{eqn:auxellpnorm}
	&&	\int_{Y \times X} f^{* \, p}(x,y)\,d(\nu \times \mu)(y,x) = \int_0^{\infty} \nu \times \mu\big( [f^{*\, p} > \lambda] \big)\,d\lambda \nonumber \\
		 && \quad =  \sum_{P \in \mathcal{P}} \int_0^{\infty} \nu \times \mu\big( (y,x) \in Y \times P:\, f^{*}(y,x) > \lambda^{1/p} \big)\, d\lambda \nonumber \\
		&& \quad \leq \sum_{P \in \mathcal{P}} \int_0^{\infty} \min\big\{ \mu(P);\, e^{-\lambda^{1/p}} \big\} \, d\lambda 
		=\sum_{P \in \mathcal{P}} \Big( \int_0^{(-\log\,\mu(P))^p} \mu(P)\,d\lambda \,+ \, \int_{(-\log\,\mu(P))^p}^{\infty} e^{-\lambda}\, d\lambda \Big) \nonumber \\
		&& \quad = \sum_{P \in \cP} \mu(P) \big( - \log \mu(P) \big)^p + \sum_{P \in \cP} e^{-(- \log \mu(P))^p}.
		\end{eqnarray}
		Moreover, we observe that 
		\begin{align*}
		\sum_{P \in \cP} e^{-(- \log \mu(P))^p} &= \sum_{P \in \cP:\, \mu(P) \geq e^{-1}} e^{-(- \log \mu(P))^p}  + \sum_{P \in \cP:\, \mu(P)< e^{-1}} e^{-(- \log \mu(P))^p}  \\
		&\leq \sum_{P \in \cP:\, \mu(P) \geq e^{-1}} e^{-(- \log \mu(P))^p} + \sum_{P \in \cP} e^{-(-\log \mu(P))} \\
		&= \sum_{P \in \cP:\, \mu(P) \geq e^{-1}} e^{-(- \log \mu(P))^p} + 1 < \infty
		\end{align*}
		for all $p \geq 1$,
		where we use the fact that the sum in the last estimate consists of finitely many terms only.  
		Now combining the assumption on $\cP$ with the inequality~\eqref{eqn:auxellpnorm} 
		proves part~(b).  
		
		\medskip
		\noindent
		We turn to the proof of the convergence assertions for the sequence $(f_n)$ as stated in part~(c). We define 
	\[
	f^{\infty}(y,x) := \mathcal{J}\Big(\cP\,|\, \bigvee_{\eta \in \cF_{\infty}^1(y)} \alpha(y,\eta)\cP \Big)(x)\,.
	\]  
	We now claim that for $\nu$-almost every $y \in Y$, there is some conull subset $\widetilde{X}_y \subseteq X$ such that 
	\begin{eqnarray} \label{eqn:martingale}
	\lim_{n \to \infty} f_n(y,x) = f^{\infty}(y,x)
	\end{eqnarray}
	for all $x \in \widetilde{X}_y$. Indeed, for each $P \in \cP$ and for $\nu$-almost all $y \in Y$, the function \[
	m^{y,P}_n(x) := \mu(P\,|\, \bigvee_{\eta \in \cF^1_n(y)} \alpha(y,\eta)\cP)(x) 
	\] 
	is an $L^1$-bounded martingale, since the sets $\cF_n^1(y)$ are nested.
	Thus, by the martingale convergence theorem, we see that 
	\begin{eqnarray} \label{eqn:martingaleaux}
	\lim_{n \to \infty} m_n^{y,P}(x) = m^{y,P}_{\infty}(x) := 
	\mu\Big(P\,|\, \bigvee_{\eta \in \cF^1_{\infty}(y)} \alpha(y,\eta)\cP\Big)(x)
	\end{eqnarray}
	for all $y \in Y$ and for $\mu$-almost every $x \in X$. We set
	$r^{y,P}_n(x):= -\log m_n^{y,P}(x) \cdot \one_P(x)$.
	Since $\mu(P) > 0$ the latter inequality is well-defined for $\mu$-almost all $x \in X$ and all $n \in \NN$ by the definition of the conditional expectation (using also the convention $-0 \cdot \log 0 = 0$ for $x \notin P$).
	  Note that  $0 \leq r^{y,P}_n(x) \leq f_n(y,x)$ for all $(y,x) \in Y \times X$. Moreover, using the assertion~(b) in the case $p=1$ (which amounts to finite Shannon entropy) we see that for $\nu$-almost all $y \in Y$ and for $\mu$-almost all $x \in X$, we have $\limsup_{n \to \infty} r_n^{y,P}(x) < \infty$. 
	 Hence, if $x \in P$ then the limit $m_{\infty}^{y, P}$ in~\eqref{eqn:martingaleaux} must be positive. 
	Now the continuity of the function
	$g(t) = -\log t$ yields
	\[
	\lim_{n \to \infty} r_n^{y,P}(x)  = -\log \mu\Big( P\,|\, \bigvee_{\eta \in \cF_{\infty}^1(y)} \alpha(y,\eta) \cP \Big)(x) \cdot \one_P(x)
	\]
	for $\nu$-almost all $y \in Y$ and all $x \in \widetilde{X}_y$, where $\widetilde{X}_y \subseteq X$ is a measurable set with $\mu(\widetilde{X}_y) = 1$. Summing over all atoms of $\cP$  leads to the equality~\eqref{eqn:martingale} which in turn implies the convergence of $f_n$ towards $f^{\infty}$ as $n \to \infty$ for $(\nu \times \mu)$-almost every $(y,x) \in Y \times X$. The $L^1$-convergence is a consequence of the dominated convergence theorem which is applicable due to the assertion~(b) in the case $p=1$. 
\end{proof}

With the above preparations at our disposal, we can now prove the stated Shannon-McMillan-Breiman theorem. The proof is generalization of the classical proof, see e.g.\@ \cite[Thm.\@ 6.2.3]{Pe83}, adapted to the present context.

\begin{proof}[Proof of Theorem~\ref{thm:SMBmain}]
Similarly as in the above lemma, we set 
\[
f_n(y,x) := \mathcal{J}\big( \mathcal{P}\,|\, \bigvee_{j=1}^n  \alpha(y,S^jy)\cP \big)(x)
\]
for $n \in \NN$ and $(y,x) \in Y \times X$. Using the general formulas
\[
\mathcal{J} \big( \mathcal{P} \vee \mathcal{Q} \big)(x) = 
\mathcal{J} \big( \mathcal{P} \big)(x) + \mathcal{J}  \big( \mathcal{Q} \, |\, \mathcal{P} \big)(x) \quad \mbox{and} \quad \mathcal{J}(g^{-1}\mathcal{P})(x) = \mathcal{J}(\cP)(gx) 
\]
for all partitions $\cP, \cQ$, $g \in \Gamma$ and almost every $x \in X$ (see \cite[Ch.~5.2]{Pe83}), we compute
\begin{eqnarray*}
	\mathcal{J}(\bigvee_{j=0}^n \alpha(y,S^jy)\cP)(x) &=& \mathcal{J}(\bigvee_{j=1}^n \alpha(y,S^jy)\cP)(x) + \mathcal{J}\big( \mathcal{P}\,|\, \bigvee_{j=1}^n  \alpha(y,S^jy)\cP \big)(x) \\
	&=& \mathcal{J}\big( \alpha(y,Sy) \bigvee_{j=1}^n \alpha(Sy,y)\alpha(y, S^j y) \cP \big)(x)	  + \mathcal{J}\big( \mathcal{P}\,|\, \bigvee_{j=1}^n  \alpha(y,S^jy)\cP \big)(x) \\
	&=& \mathcal{J}(\bigvee_{j=1}^{n} \alpha(Sy,S^{j-1} Sy)\cP)(\alpha(Sy, y)x) + \mathcal{J}\big( \mathcal{P}\,|\, \bigvee_{j=1}^n  \alpha(y,S^jy)\cP \big)(x) \\
	&=& \mathcal{J}(\bigvee_{j=0}^{n-1} \alpha(Sy,S^{j} Sy)\cP)(\alpha(Sy, y)x) + \mathcal{J}\big( \mathcal{P}\,|\, \bigvee_{j=1}^n  \alpha(y,S^jy)\cP \big)(x),
	\end{eqnarray*}
which by induction on $n$ yields
\begin{eqnarray*}
	\mathcal{J}(\bigvee_{j=0}^n \alpha(y,S^jy)\cP)(x) &=& \sum_{j=0}^n f_{n-j}\big( T^j(y,x) \big),
	\end{eqnarray*}
for all $n \in \NN_{\geq 0}$ and almost every $(y,x)$.
We define $f^{\infty}(y,x) := \mathcal{J}\Big( \cP\,|\, \bigvee_{j=1}^{\infty} \alpha(y,S^j y)\cP \Big)(x)$.  It now follows from Lemma~\ref{lemma:integrability} (part ~(b) (for $p=1$) and part (c)) that 
\begin{eqnarray} \label{eqn:convinlemma}
f^{\infty}(y,x):= \mathcal{J}\Big( \cP\,\big|\, \bigvee_{j=1}^{\infty} \alpha(y,S^jy) \cP \Big)(x) = \lim_{n \to \infty} f_n(y,x)
\end{eqnarray}
both in $L^1(Y \times X, \nu \times \mu)$, as well as $(\nu \times \mu)$-almost everywhere. We write
\begin{eqnarray*}
\frac{1}{n+1} \mathcal{J}\Big( \bigvee_{j=0}^n \alpha(y,S^j y) \cP \Big)(x) &=& \frac{1}{n+1} \Bigg( \sum_{j=0}^n f^{\infty}\big( T^j(y,x) \big)
+  \sum_{j=0}^n \big( f_{n-j}(T^j(y,x)) - f^{\infty}(T^j(y,x)) \big) \Bigg).
\end{eqnarray*}
It is an immediate consequence from Birkhoff's pointwise ergodic theorem that 
\begin{eqnarray*}
\lim_{n \to \infty} \frac{1}{n+1} \sum_{j=0}^n f^{\infty}(T^j(y,x)) &=& \int_{Y \times X} f^{\infty}(y,x)\,d(\nu \times \mu)(y,x) \\
&=&  \int_{Y \times X} \mathcal{J}\Big( \cP\,|\, \bigvee_{j=1}^{\infty} \alpha(y,S^j y)\cP \Big)(x) \,d(\nu \times \mu)(y,x) \\
&=& \int_Y H\Big( \cP\,|\, \bigvee_{j=1}^{\infty} \alpha(y,S^j y) \cP \Big)\, d\nu(y) \\
\mbox{ (Proposition~\ref{prop:reporbitalentropy}) }&=& \mathfrak{h}_\mathcal{P}(\alpha, X)
\end{eqnarray*}
for $(\nu \times \mu)$-almost all $(y,x)$. Moreover, this convergence also holds in $L^1(Y \times X, \nu \times \mu)$. Hence, the proof of the theorem is completed once we show that 
\begin{eqnarray} \label{eqn:remainstoshow}
\lim_{n \to \infty} \frac{1}{n+1} \sum_{j=0}^n \big| f_{n-j}(T^j(y,x)) - f^{\infty}(T^j(y,x)) \big| = 0
\end{eqnarray}
for $(\nu \times \mu)$-almost every $(y,x)$ and in $L^1(Y \times X, \nu \times \mu)$. To this end, we fix $N \in \NN$ and set $F_N := \sup_{k \geq N} |f_k - f^\infty|$. 
We compute
\begin{eqnarray} \label{eqn:aux1}
&& \frac{1}{n+1} \sum_{j=0}^n \big| f_{n-j}(T^j(y,x)) - f^{\infty}(T^j(y,x)) \big| \nonumber \\
&& \quad = \frac{1}{n+1}\sum_{j=0}^{n-N} \big| f_{n-j}(T^j(y,x)) - f^{\infty}(T^j(y,x)) \big| + \frac{1}{n+1}\sum_{j=n-N+1}^n \big| f_{n-j}(T^j(y,x)) - f^{\infty}(T^j(y,x)) \big| \nonumber \\
&& \quad \leq  \frac{1}{n+1}\sum_{j=0}^{n-N}  F_N(T^j(y,x)) + \frac{1}{n+1} \sum_{j=n-N+1}^n \big| f_{n-j}(T^j(y,x)) - f^{\infty}(T^j(y,x)) \big|.
\end{eqnarray}
Now $|f_{n-j} - f^{\infty}| \leq f^{*} + f^{\infty}$ (where $f^{*} = \sup_{n \geq 0} f_n$) almost surely,  and we get from Lemma~\ref{lemma:integrability} part~(b) (for the case $p=1$) that 
\[
f^{*} + f^{\infty}~\in~L^1(Y \times X, \nu \times\mu).
\]
Further, since one sums up only finitely many (precisely $N$) terms, the second summand in~\eqref{eqn:aux1} tends to $0$ by the Birkhoff ergodic theorem. 
On the other hand, we have $F_N \leq f^{*} + f^{\infty} \in L^1(Y \times X, \nu \times \mu)$. Applying Birkhoff's ergodic theorem to the 
first summand of~\eqref{eqn:aux1} yields
\begin{eqnarray*}
\limsup_{n \to \infty} \frac{1}{n+1} \sum_{j=0}^{n-N} |f_{n-j}(T^j(y,x)) - f^{\infty}(T^j(y,x))| &\leq& \lim_{n \to \infty} \frac{1}{n+1} \sum_{j=0}^{n-N}
F_N(T^j(y,x)) \\ 
&=& \int_{Y \times X} F_N\, d(\nu \times \mu).
\end{eqnarray*}
for all $N \in \NN$. Since $(f_k)$ converges to $f^{\infty}$ almost surely we have 
\[
\lim_{N \to \infty} F_N(y,x) = 0
\]
for $(\nu \times \mu)$-almost every $(y,x)$. The dominated convergence theorem now gives
\[
\lim_{N \to \infty}  \int_{Y \times X} F_N\, d(\nu \times \mu) = 0.
\]
So we have verified both summands of~\eqref{eqn:aux1} to converge to $0$ pointwise almost surely
as $n \to \infty$. This shows the convergence in \eqref{eqn:remainstoshow} almost surely. Moreover, we observe that since $(f_n)$ converges to $f^{\infty}$ in $L^1$, we get 
\begin{align*}
\limsup_{n \to \infty}  \frac{1}{n+1} \sum_{j=0}^{n} \big\| f_{n-j} \circ T^j - f^{\infty} \circ T^j \big\|_{L^1} = \limsup_{n \to \infty}  \frac{1}{n+1} \sum_{j=0}^{n} \big\| f_j - f^{\infty} \big\|_{L^1} = 0,
\end{align*}
which is $L^1$ convergence in~\eqref{eqn:remainstoshow}. 
This concludes the proof of the theorem.
\end{proof}

\subsection{Integrated SMB theorems}
With additional integrability assumption on the partition $\cP$ one obtains the following entropy convergence statements. 

\begin{Corollary} \label{cor:SMBint}
	In the situation of Theorem~\ref{thm:SMBmain}, suppose that there is some $\varepsilon_0 > 0$ such that $\sum_{P \in \cP} \mu(P)\big( - \log \mu(P) \big)^{1 + \varepsilon_0} < \infty$. Then the following assertions hold. 
	\begin{enumerate}[(a)]
		\item For $\nu$-almost all $y \in Y$, the normalized information functions
		\[
		 \frac{1}{n+1}\mathcal{J} \Big( \bigvee_{j=0}^n \alpha(y,\, S^j y)\cP \Big)(\cdot)
		\]
		 converge to $\mathfrak{h}_{\cP}(\alpha, X)$ in $L^1(X,\mu)$ as $n \to \infty$. 
		 	\item The convergence
		 \[
		 \mathfrak{h}_{\cP}(\alpha, X) = \lim_{n \to \infty} \frac{1}{n+1} \int_Y \mathcal{J} \Big( \bigvee_{j=0}^n \alpha(y,\, S^j y)\cP \Big)(x)\, d\nu(y)
		 \]
		 holds for $\mu$-almost all $x \in X$ and in $L^1(X,\mu)$. 
	\end{enumerate}
\end{Corollary}

\begin{proof}
	As before, we set
	\[
	f_n(y,x) := \mathcal{J}\big( \cP\,|\, \bigvee_{j=1}^n \alpha(y, S^jy) \cP \big)(x), \quad f_0(y,x) := \mathcal{J}\big( \cP \big)(x)
	\]
	and $f^{*} := \sup_{n \geq 0} f_n$. 
	As in the proof of Theorem~\ref{thm:SMBmain} we note that 
	\[
	\frac{1}{n+1}\mathcal{J} \Big( \bigvee_{j=0}^n \alpha(y, S^j y) \cP \Big)(x) =  \frac{1}{n+1} \sum_{j=0}^n f_{n-j}\big( T^j(y,x) \big)
	\]
	for all $n \in \NN$. This yields 
	\[
	\sup_n  \frac{1}{n+1}\mathcal{J} \Big( \bigvee_{j=0}^n \alpha(y, S^j y) \cP \Big)(x) \leq \sup_n \frac{1}{n+1} \sum_{j=0}^n f^{*}\big( T^j(y,x) \big) = M f^{*}(y,x),
	\]
	where $M:L^1(Y \times X) \to L^0(Y \times X), (M k)(y,x) := \sup_{n \in \NN} \frac{1}{n} \big| \sum_{j=0}^{n-1} k\big(T^j(y,x) \big)\big|$ denotes the maximal operator with respect to the p.m.p.\@ transformation $T$ on $(Y \times X,\, \nu \times \mu)$. Using the additional integrability condition on $\cP$ we find that $f^{*} \in L^{1 + \varepsilon_0}(Y \times X, \nu \times \mu)$
	by assertion~(b) of Lemma~\ref{lemma:integrability}. Standard theory now gives $Mf^{*} \in L^{1 + \varepsilon_0}(Y \times X,\, \nu \times \mu)$ as well, cf.\@ \cite[Thm.\@ 8, p. \@678]{DS}. Since we are dealing with finite measure spaces we observe that $Mf^{*} \in L^1(Y \times X)$. In particular, 
	for $\nu$-almost every $y \in Y$, we obtain $\int_X Mf^{*}(y,x)\, d\mu(x) < \infty$ and 
	for $\mu$-almost every $x \in X$, we get $\int_Y Mf^{*}(y,x)\,d\nu(y) < \infty$.  Consequently,
		\[
	\int_X \sup_n \frac{1}{n+1} \mathcal{J} \Big( \bigvee_{j=0}^n \alpha(y, S^j y) \cP \Big)(x)\, d\mu(x) < \infty	
	\]
	for $\nu$-almost every $y \in Y$, as well as 
	\[
	\int_Y \sup_n \frac{1}{n+1} \mathcal{J} \Big( \bigvee_{j=0}^n \alpha(y, S^j y) \cP \Big)(x)\, d\nu(y) < \infty 
	\]
	for $\mu$-almost every $x \in X$.
	Hence, by Theorem~\ref{thm:SMBmain} and the Lebesgue dominated ergodic theorem we obtain both desired statements. 
	\end{proof}

\section{Proofs of the main results}\label{sec-proof-main}
In this section we prove the main results stated above about pointwise entropy equipartition along families  of finite sets in the groups under consideration.
 
\subsection{Entropy convergence along almost-geodesic segments}
\subsubsection{Shannon-McMillan-Breiman theorem in the free group} \hfill

\noindent {\it Proof of  Theorem \ref{free-gps-SMB}.} For the free group, we take as our auxiliary set $Y$ the set $D$ of bi-infinite geodesics passing through the identity. The space $D$ can be identified with the set $D^\prime$ consisting of the points 
$\tilde{\xi}=(\cdots, \xi_{-m}, \xi_{-m+1},\cdots ,\xi_{-1},\xi_0,\xi_1,\cdots \xi_n,\cdots)$ with $\xi_i \xi_{i+1}\neq e$, where $\xi$ belongs to the symmetric free generating sets for all $i \in \ZZ$. We denote by $S$ the forward shift on the space of bi-infinite reduced words $\tilde{\xi}$, given by $S(\tilde{\xi})(i)=\xi_{i+1}$. The standard Markov measure on $\partial \FF_r$ is denoted by $\nu_{PS}$. The boundary $(\partial \FF_r,\nu_{PS})$ is a factor space of the space  $D^\prime\cong D\subset \partial^2 \FF_r$ of bi-infinite geodesics passing through $e$. The factor map is given by $\tilde{\xi}=(\xi^{-}, \xi^{+})=((\xi_{-j})_{j\in \NN_+},(\xi_i)_{i\in \NN_{\ge 0}}) \mapsto \xi^{+}$, and is equivariant with respect the forward shift maps on both spaces. Furthermore, the (properly normalized) $S$-invariant measure $\nu_{PS}\times \nu_{PS}$ on $D$ is mapped to the standard (=equal subdivison) shift-invariant Markov probability measure on $D^\prime$, which in turns is mapped to  $\nu_{PS}$ on $\partial \FF_r$.  
 
We define the map the $\alpha(S^n\tilde{\xi}, \tilde{\xi})$ (for $n > 0$ and correspondingly for $n\le 0$) to be  $( \xi_0\xi_1\cdots \xi_{n-1})^{-1}$, with $\tilde{\xi}$ as defined above. 
 It is easily checked that this map is in fact a cocycle of the equivalence relation defined by the $S$-orbits on $D^\prime$. Referring to Remark \ref{notation convention}, we have 
 $\cF_n(\tilde{\xi})=\set{S^k\tilde{\xi}\,;\, 0\le k \le n}$ and so 
 $$F_n(\tilde{\xi})=\set{\alpha(\tilde{\xi},S^k\tilde{\xi})\,;\, 0 \le k \le n}=\set{e}\cup \set{\xi_0, \xi_0\xi_1,\dots,\xi_0\xi_1\cdots\xi_{n-1}}\,.$$
 Therefore writing $\tilde{\xi}^+=(\xi^-, \xi^+)$, denoting  $\xi^+$ by $\xi\in \partial \FF_r$, the  refined partition are given by 
 $$\cP\vee\xi_0\cP\vee \xi_0\xi_1\cP\vee\cdots \vee \xi_0\xi_1\cdots\xi_{n-1} \cP\,.$$ 
 
The  cocycle $\alpha$ defined on the relation $\cR=\cO_S$ is class-injective since the transformation $S$ is essentially free. Furthermore, the fact that the extended action is ergodic follows from the fact that the two-sided shift is strong mixing. 

The proof of Theorem \ref{free-gps-SMB} is now complete upon appealing to  Theorem  \ref{thm:SMBmain}, and noting that under the map $\tilde{\xi}=(\xi^{-}, \xi^{+})\mapsto \xi^{+}$ the law of $\xi^{+}$ is indeed equivalent to $\nu_{PS}$, as stated.  
\qed

We note that a more general point of view is given by the discussion in \cite[\S~6.2.2]{BN15a}.  There an orbital  cocycle  
$\alpha^\prime$ with values in $\FF_r$ is constructed, which is defined on the orbit relation $\cR$ of the $\FF_r$-action of the space $\partial^2 \FF_r$ restricted to the set $D$ of bi-infinite geodesics passing through the identity. It is shown there that the space $D$ can be identified with the set $D^\prime$ above, and the orbits of the shift $S$ on $D^\prime$ can be identified with the equivalence classes of the relation $\cR$. Furthermore, under these identification, the cocycle $\alpha^\prime$ is in fact equal to 
$\alpha(S^n\tilde{\xi},\tilde{\xi})$. (Note however that there is a typo in 
\cite[\S~6.2.2]{BN15a}, and $T: D^\prime\to D^\prime$ defined there should be the forward shift, not the backward shift). Class injectivity of the cocycle $\alpha$ follows since the $\FF_r$-action on $\partial^2 \FF_r$ is essentially free.

\medskip
\noindent {\it Proof of  Corollary~\ref{free-gps-SMB-int}.} We keep the notation of the proof of Theorem~\ref{free-gps-SMB}. A short computation (using also that $S_n^{-1}=S_n$) shows that 
\[
\int_{\partial\FF_r} \frac{\mathcal{J}\Big( \bigvee_{j=0}^n \alpha(\xi, S^j \xi) \cP \Big)(x)}{n+1}\,d\nu(\xi) = \frac{1}{|S_n|} \sum\limits_{\substack{g \in S_n, \\ g = a_1a_2 \cdots a_n}} \frac{\mathcal{J}\Big( \cP \vee \bigvee_{j=1}^n (a_1 \cdots a_j)\cP \Big)(x)}{n+1}
\]
for all $n \in \NN$ and $x \in X$.
Therefore, an application of part~(b) of Corollary~\ref{cor:SMBint} completes the proof.
\qed  
\subsubsection{Shannon-McMillan-Breiman theorem for hyperbolic groups}

For negatively curved groups and lattice subgroups of real-rank-one Lie groups,  
the auxiliary space $(Y,\nu)$, the transformation $S : Y \to Y$ (and also its orbit relation $\cR$), and the cocycle $\alpha : \cR \to \Gamma$ were described explicitly in \S~\ref{sec:RManifold-1}, \ref{sec:Rlocallysymmetricspaces}. 
The fact that the cocycle is weak-mixing follows from the well-known fact that the geodesic flow is strongly mixing for negatively curved groups and lattice subgroups. This follows from the Hopf argument  (see e.g.\@ \cite[Thm.\@ 17.6.2]{KH95}), and the Howe-Moore theorem in the second. 

Turning to the case of word-hyperbolic groups, weak-mixing of the cocycle for the $\RR$-flow was proved in  \cite[Cor. 1.16]{BF22}, as noted in \S~\ref{sec:hyp}.  
Here however we encounter a subtlety, since weak mixing for the $\RR$-action does not imply that for every non-zero $t\in \RR$ the cocycle restricted to the action of $\ZZ\cdot t$ is weak-mixing. In fact, it is stated in \cite{BF22} (see the discussion preceding Cor. 1.14) that the geodesic flow may have a Kronecker factor $\RR/(\ZZ\cdot u_0)$. In that case, some elements $0\neq u\in \RR$ will have the property that the subgroup $\ZZ\cdot u$ is not ergodic, and the cocycle restricted to it is not weak-mixing. But for every ergodic $\RR$-action, for almost every $t$ the action 
restricted to $\ZZ\cdot t$ is ergodic. Choosing $t_0$ to be such a non-exceptional, namely ergodic element, we conclude that the sets $\cF_n(y)$, which were defined along the orbit of a discrete subgroup $\ZZ\cdot t_0$ in \S~\ref{sec:hyp}, will depend on the action $X$, and more precisely, on the subset of $\RR$ consisting of elements $t$ such that $\ZZ\cdot t$ is ergodic on the extension $S^1M\times X$.  However, the sets $\cF_n(y)$ do not depend on the action whenever the geodesic flow is strongly mixing. 

Appealing to the convergence results stated in Theorem \ref{thm:SMBmain}, we conclude that for non-exceptional elements the convergence results stated in Theorems 
\ref{neg-curv--SMB}, \ref{lattice-sbgs}  and \ref{SMB-hyp-gps} hold. 

Finally, let us consider the distribution of the limit point $F^+(y)$. In the case of negatively curved groups discussed in \S~\ref{sec:RManifold-1}, there is a well-known correspondence between geodesic-flow-invariant probability measures on $S^1 M$, infinite Radon measures on $S^1 \widetilde{M} $ invariant under both the fundamental group $\Gamma$ and the geodesic flow, and $\Gamma$-invariant infinite Radon measures on $\partial^2 \widetilde{M}$, see e.g. \cite[\S~2.1, Thm.\@ 2.2]{Ka90} for a full discussion.  
The correspondence arises by identifying $S^1 \widetilde{M} $ with 
$\partial^2 \widetilde{M}\times \RR$ using the Hopf coordinates, and noting that  the mapping $\partial^2 \widetilde{M}\times \RR$ onto $\partial^2 \widetilde{M}$ is the map assigning to a bi-infinite geodesic its two end-points. Given $y=(p,v)\in S^1 M=Y$, $\cF^+(y)$ is simply the forward limit point of the geodesic starting at $\widetilde{p}\in \widetilde{D}_1$ with direction $v$. This follows from the fact that $F_n((p,v))$ is a $C$-almost geodesic with this limit point. Choosing as the measure on $S^1 M$ the measure $\mu^{BM}$, the projection described above on $\partial^2 \widetilde{M}$ gives a measure equivalent to $\nu_{PS}\times\nu_{PS}$ (see \cite[\S~3.5]{Ka90}). Therefore on $\partial \widetilde{M}$ the resulting measure is equivalent to $\nu_{PS}$. 

In the case of general hyperbolic groups discussed in \S~\ref{sec:hyp}, consider the probability measure denoted there by $\mu^{BM}$, which is invariant under the geodesic flow on $S^1 M$. By construction it arises from the identification of $S^1 M$ with a measurable bounded fundamental domain $\cD_H$ for the $H$-action on $\partial^2 H\times \RR$, using a measure equivalent to $\nu_{PS}\times \nu_{PS}\times L$, restricted to to $\cD_H$, where $\nu_{PS}$ belongs to the Patterson-Sullivan measure class  (this is in analogy with $\partial^2 \widetilde{M}\times \RR$ for negatively curved group). The limit point $ F^+(y)$ arises here from the projection $(\xi^-,\xi^+,t)\to \xi^+$, and hence its law is in the Patterson-Sullivan measure class in the this case as well. 
This completes the proofs of Theorem
\ref{neg-curv--SMB}, Theorem \ref{lattice-sbgs}  and Theorem \ref{SMB-hyp-gps}. 

\qed 

 Finally recall that in Proposition \ref{thm:ENT} it was shown that under the special assumptions stated there, namely freeness, injectivity and ergodicity 
 $\mathfrak{h}^{orb}(\alpha, X) = h^{\operatorname{Rok}}\big(\Gamma \curvearrowright X\big).$ To complete the picture, we now state : 

\begin{Corollary}
The allowable choices of $t_0$ such that the special assumptions are satisfied are as follows. 
\begin{itemize} 
\item For negative curved groups : any $t_0\neq 0$ for negatively curved groups which satisfies $\abs{t_0} > 3\, \mathrm{Diam}(\cD)$. 
\item For lattice subgroups : any  $t_0\neq 0$ such that the associated cocycle $\alpha$ is class-injective.
\item  For word hyperbolic groups : any $t_0$ which is non-exceptional and satisfies $\abs{t_0} > 3\, \mathrm{Diam}(\cD)$. 
\end{itemize} 
\end{Corollary}

\subsection{Convergence of information along trajectories of random walks }\label{sec:random}

\subsubsection{Shannon-McMillan-Breiman theorem along random walk trajectories of a general group.}

\noindent {\it Proof of  Theorem \ref{random-SMB}.} Let $(\Omega,{\bf p}) =\prod_{i\in \ZZ} (\Gamma,m)$ be the direct product of $\ZZ$ copies of $(\Gamma,\mu)$, and let $S: \Omega \to \Omega$ be the shift defined by  $S(\omega)_i=\omega_{i+1}$. Let $\cO_S=\cR$ be the orbit relation of $S$, and define a cocycle $\alpha : \cO_S \to \Gamma$ as follows :  
$\alpha(S^n\omega,\omega)=\omega_{n-1}\cdots\omega_{1}\omega_0$ if $n > 0$, $\alpha(\omega,\omega)=e$, and $\alpha(S^n\omega, \omega) =\omega_{n}^{-1} \omega_{n+1}^{-1}\cdots \omega_{-1}^{-1}$ if $n < 0$ . Then $a(k,\omega):= \alpha(S^k\omega, \omega): \ZZ\times \Omega \to \Gamma$ satisfies the cocycle identity, namely $a(k+m,\omega)=a(k, S^m\omega)a(m,\omega)$. Here $a(1,\omega)=\alpha(S\omega, \omega)=\omega_0\in \Gamma$ is the map from $\Omega$ to $\Gamma\subset \text{Aut }(X,\mu)$. 
The skew product map is given by 
$T(\omega, x)=(S\omega, \alpha(S\omega,\omega) x)=(S\omega, \omega_0 x)$, so that  for $n\ge 0$ 
$$T^n(\omega, x)=(S^n\omega,\alpha(S^n\omega, \omega)x)=(S^n\omega, \omega_{n-1}\cdots \omega_1\omega_0 x)\,.$$ 

Define 
$$F_n(\omega)=\set{\alpha(\omega, S^k\omega)\,;\, 0 \le k \le n}=\set{e,\omega_0^{-1}, \omega_0^{-1}\omega_1^{-1}, \dots, \omega_0^{-1}\omega_1^{-1}\cdots\omega_{n-1}^{-1}} $$
so that for a partition $\cP$ of $X$ we have, for $n \ge 1$ 
$$\cP^{F_n(\omega)}=\cP\vee\omega^{-1}_0\cP\vee \omega^{-1}_0\omega^{-1}_1\cP\vee\cdots\vee \omega^{-1}_{0} \omega_1^{-1}\cdots \omega^{-1}_{n-1}\cP=\cP\vee 
\bigvee_{k=0}^{n-1}\left(\omega_k\cdots\omega_0\right)^{-1}\cP\,.$$
The shift map $S$ is essentially free and strongly mixing and it follows that the cocycle $\alpha$ is
 weakly mixing and, see \S~\ref{sec:mix} below.  
Therefore 
the refined partitions satisfy the conclusions of Theorem  \ref{thm:SMBmain}, upon setting $Z_i(\omega)=\omega_i$. \qed

\section{Strongly mixing actions and weakly mixing cocycles}\label{sec:mix}
A key assumption in Theorem~\ref{thm:SMBmain} is that the skew product transformation $T$ on $(Y\times X, \nu\times \mu)$ is ergodic. In this section, we give a simple sufficient condition on the $\Gamma$-action on $X$, so that for any class-injective cocycle $\alpha: \cO_S \to \Gamma$, the skew transformation $T$ will be ergodic. This condition is thus sufficient  for the of validity Theorem~\ref{thm:SMBmain} and Corollary~\ref{cor:SMBint}. 
It turns out that for all class-injective transformation cocycles $\alpha: \cR \to \Gamma$ arising from an essentially free, ergodic transformation $S$ on $(Y, \nu)$, the induced transformation $T$ is ergodic if the action of 
$\Gamma$ on $(X,\mu)$ is  {\em strongly mixing}, see Proposition~\ref{prop:mixing} below.

\begin{Definition}[Strongly mixing action] \label{defi:mixingaction}
	A measure preserving action of a countable group $\Gamma$ on a probability space is called {\em strongly mixing} if
	\begin{eqnarray*} 
		\lim_{\gamma \to \infty} \big\langle \pi_X(\gamma)f,\, h \big\rangle_{L^2(X,\mu)} = \langle f,\, 1 \rangle_{L^2} \cdot \langle  1,\, h\rangle_{L^2}
	\end{eqnarray*}
	for all $f,h \in L^2(X,\mu)$, where $\pi_X(\gamma) f(x) := f(\gamma^{-1}x)$ and the above limit is taken w.r.t.\@ 
	an arbitrary sequence $(\gamma_n) \in \Gamma^{\NN}$ such that for all finite $F \subset \Gamma$, there is some
	$N \in \NN$ such that $\gamma_n \notin F$ for all $n \geq N$.
\end{Definition}

As is well known, for $\Gamma$-actions strong mixing implies ergodicity.

\medskip

The validity of the following proposition follows by a slight adaption of the proof of Proposition~5.3 in \cite{NP17}. The argument used there was brought to the authors' attention by A. Kechris.  

\begin{Proposition}\label{prop:mixing}

	Let $\Gamma$ be a countable group acting on a standard probability space $(X,\mu)$, preserving the measure $\mu$.
	Suppose that 
	$\alpha: \cO_S \to \Gamma$
	is a class-injective  cocycle arising from an essentially free, probability
	measure preserving transformation $S$ on $(Y,\nu)$. 
	If 
	\begin{itemize}
		\item the action of $\Gamma$ on $(X,\mu)$ is strongly mixing,
		\item and the action of $S$ on $(Y,\nu)$ is ergodic,
	\end{itemize}
	then the induced transformation
	\[
	T: (Y \times X,\, \nu \times \mu) \to (Y \times X,\, \nu \times \mu), \,\, T(y,x) := \big( Sy,\, \alpha(Sy,y)x \big)
	\] 
	is ergodic.
\end{Proposition}

We obtain the following general Shannon-McMillan-Breiman theorem for strongly mixing p.m.p.\@ group transformations. 

\begin{Corollary}[Entropy equipartition for strongly mixing actions] \label{cor:SMBmixing}
	Let $\Gamma$ be a countable group. We assume that $\alpha: \cR\to \Gamma$ is a class-injective cocycle defined on the orbit relation $\cR=\cO_S$ of  a measure preserving, essentially free, invertible, ergodic transformation $S$ on a probability space $(Y,\nu)$. 
	Then, for every strongly mixing p.m.p.\@ action of $\Gamma$ on a probability space $(X,\mu)$ and for every countable
	partition $\cP$ of $X$ with $H(\cP) < \infty$, we have
	\[
	\mathfrak{h}_\cP(\alpha,X) = \lim_{n \to \infty} \frac{1}{n+1} \mathcal{J}\Big( \bigvee_{j=0}^n \alpha(y, S^j y)\cP \Big)(x) 
	\]
	for $(\nu \times \mu)$-almost every $(y,x) \in Y \times X$ and in $L^1(Y \times X, \nu \times \mu)$. \\ 
	If there is some $\varepsilon_0 > 0$ such that  $\sum_{P \in \cP} \mu(P) (-\log \mu(P))^{1 + \varepsilon_0} < \infty$, then the above convergence also holds in $L^1(X,\mu)$ for $\nu$-almost every $y \in Y$ and also
	\[
	\mathfrak{h}_{\cP}(\alpha, X)= \lim_{n \to \infty} \frac{1}{n+1} \int_Y \mathcal{J}\Big( \bigvee_{j=0}^n \alpha(y, S^j y) \cP \Big)(x)\,d\nu(y) 
	\]
	for $\mu$-almost all $x \in X$. 
\end{Corollary}

\begin{proof}
	It follows from Proposition~\ref{prop:mixing} that the extended transformation $T$ on $(Y \times X,\, \nu \times \mu)$ is 
	ergodic. 
	The statement about $\nu \times \mu$-almost everywhere convergence and convergence in $L^1(Y \times X, \nu \times \mu)$ now follows from Theorem~\ref{thm:SMBmain}. With the stronger integrability condition on $\cP$ at our disposal we can use Corollary~\ref{cor:SMBint} in order to get the remaining convergence assertions.
\end{proof}

\end{document}